\documentclass{aims}

\usepackage[latin1]{inputenc}
\usepackage[T1]{fontenc}
\usepackage[american]{babel} 
\usepackage{comment}

 \usepackage[colorlinks=true]{hyperref}
\hypersetup{urlcolor=blue, citecolor=red}

\def\currentvolume{X}
 \def\currentissue{X}
  \def\currentyear{200X}
   \def\currentmonth{XX}
    \def\ppages{X--XX}
     \def\DOI{10.3934/xx.xx.xx.xx}

\usepackage{amsfonts,amssymb,amsthm,amsopn}
\usepackage[centercolon]{mathtools}
\usepackage{fixmath}
\usepackage{braket}

\makeatletter
\def\th@plain{%
	\thm@notefont{}
	\itshape 
}
\def\th@definition{%
	\thm@notefont{}
	\normalfont 
}

\usepackage[algo2e,ruled,linesnumbered]{algorithm2e}
\SetKw{KwTerminate}{terminate}
\usepackage[svgnames]{xcolor}
\usepackage{graphicx}
\usepackage{subcaption}
\usepackage{pgfplots}
\usepackage{pgfgantt}
\usepackage{pdflscape}
\pgfplotsset{compat=newest}
\pgfplotsset{plot coordinates/math parser=false}
\usetikzlibrary{external}
\usepackage{nameref}
\usepackage{xspace}
\usepackage{dlfltxbcodetips}
\usepackage{float}
\usepackage[numbers,sort&compress]{natbib}
\usepackage{rcs}
\usepackage{enumerate}
\usepackage{url}
\usepackage{placeins}
\usepackage{cleveref}

\usepackage[pagewise]{lineno}

\newtheorem{theorem}{Theorem}[section]
\newtheorem{corollary}{Corollary}

\newtheorem{lemma}[theorem]{Lemma}

\theoremstyle{definition}
\newtheorem{definition}[theorem]{Definition}
\newtheorem{remark}{Remark}

\newtheorem{assumption}[theorem]{Assumption}

\DeclareMathOperator{\supp}{supp}
\newcommand{\Lin}{\ensuremath {{\mathcal{L}}} }
\newcommand{\BV}{\ensuremath {BV(\Omega)} }
\newcommand{\M}{\ensuremath {\mathcal{M}(\Omega)} }
\newcommand{\Lone}{\ensuremath {L^1(\Omega)} }  
\newcommand{\LL}{\ensuremath {L^2(\Omega)} }  
\newcommand{\Linf}{\ensuremath {L^\infty(\Omega)} }
\newcommand{\Hoz}{\ensuremath {H_0^1(\Omega)} } 
\newcommand{\Htwo}{\ensuremath {H^2(\Omega)} } 
\newcommand{\Woneinf}{\ensuremath {W^{1,\infty}(\Omega)} } 
\newcommand{\Norm}[2]{\ensuremath{\left\|#1\right\|_{#2}}}
\newcommand{\norm}[2]{\ensuremath{\lVert #1\rVert_{#2}}}
\newcommand{\Abs}[1]{\ensuremath{\left| #1\right|}}
\newcommand{\abs}[1]{\ensuremath{| #1 |}}
\newcommand{\N}{\mathbb{N}}
\newcommand{\R}{\mathbb{R}}
\newcommand{\cala}{\ensuremath{{\mathfrak{a}}}}

\title[Error est. for control by BV functions]{Finite element error estimates for one-dimensional elliptic optimal control by BV functions}

\author[Dominik Hafemeyer, Florian Mannel, Ira Neitzel and Boris Vexler]{}


 \email{dominik.hafemeyer@tum.de}
 \email{florian.mannel@uni-graz.at}
 \email{neitzel@ins.uni-bonn.de}
 \email{vexler@ma.tum.de}
 
\thanks{The first author is supported by the Studienstiftung des deutschen Volkes.}

\thanks{$^*$ Corresponding author: Dominik Hafemeyer}

\begin{document}

\centerline{\scshape Dominik Hafemeyer$^*$}
\medskip
{\footnotesize
\centerline{Department of Mathematics, Technische Universität München,}
\centerline{ Boltzmannstr.~3, 85747 Garching b. München, Germany}
}
\medskip

\centerline{\scshape Florian Mannel}
\medskip
{\footnotesize
\centerline{Institute of Mathematics and Scientific Computing, University of Graz,}
\centerline{Heinrichstr. 36, 8010 Graz, Austria}
}
\medskip

\centerline{\scshape Ira Neitzel}
\medskip
{\footnotesize
\centerline{Institute for Numerical Simulation, Universität Bonn,}
\centerline{Endenicher Allee 19b, 53115 Bonn, Germany}
}
\medskip

\centerline{\scshape Boris Vexler}
\medskip
{\footnotesize
\centerline{Department of Mathematics, Technische Universität München,}
\centerline{ Boltzmannstr.~3, 85747 Garching b. München, Germany}
}

\bigskip

 \centerline{(Communicated by the associate editor name)}

\bigskip

\begin{abstract}
	We consider an optimal control problem governed by a one-di\-men\-sio\-nal 
	elliptic equation that involves univariate functions
	of bounded variation as controls. For the discretization of the state 
	equation we use linear finite elements and for the control discretization 
	we analyze two strategies. First, we use variational discretization of 
	the control and show that the $L^2$- and $L^\infty$-error for the state 
	and the adjoint state are of order ${\mathcal O}(h^2)$ and that the 
	$L^1$-error of the control behaves like ${\mathcal O}(h^2)$, too. 
	These results rely on a structural assumption that
	implies that the optimal control of the original problem is piecewise 
	constant and that the adjoint state has nonvanishing first derivative at the jump 
	points of the control. 
	If, second, piecewise constant control discretization is used, we obtain 
	$L^2$-error estimates of order $\mathcal{O}(h)$ for the state and $W^{1,\infty}$-error estimates of order $\mathcal{O}(h)$ for the adjoint state.
	Under the same structural assumption as before we derive an $L^1$-error estimate of order $\mathcal{O}(h)$ for the control.
	We discuss optimization algorithms and provide numerical results for both discretization schemes indicating that the error estimates are optimal.
\end{abstract}

\maketitle


\section{Introduction}\label{sec_intro}

In this paper we derive a priori error estimates for two finite element discretizations of the optimal control problem governed by a one-dimensional elliptic equation
\begin{equation*}
\min_{(u,q)} \, \frac12\Norm{u-u_d}{\LL}^2 +
\alpha\Norm{q'}{\M}
\qquad\text{s.t. }\qquad
A u = q.
\end{equation*}
Here, $u\in V:= H^1_0(\Omega)$ is the state and $q\in Q := \BV$ is the control, where $\BV$ denotes the space of functions of bounded variation (BV) on the interval $\Omega:=(0,1)$. The operator $A$ is elliptic and $\alpha$ is a positive real number. 
The two finite element schemes that will be analyzed are identical in regard to the discretization of state and adjoint state, but they differ in the treatment of the control. In the \emph{variational discretization} the control is not discretized, while in the second scheme the control is discretized by piecewise constant functions.

The significance of the above control problem is given by the use of the BV-seminorm $\|q'\|_{\mathcal{M}(\Omega)}$ in the objective. 
This favors piecewise constant controls with only a limited number of jumps, which makes this problem type interesting in many practical applications.  
The precise functional analytic setting will be provided in the next section.

Optimal control problems with BV-controls defined in one space dimension are strongly related to control problems with measures as controls.
Both BV optimal control problems and optimal control problems with measures have attracted significant research interest in the recent past, see, e.g. \cite{ClasonKunisch2010,CK2018,CasasKogutLeugering,Florian2017,Florian2018,Bredies2019} for the former and \cite{CCK2012,CCK2013,CK2014,CK2016,HQ2019,KPV2014} for the latter.

Error estimates for PDE-constrained optimal control problems involving measures have been presented in \cite{CCK2012,KPV2014,KTV2016,PV2013,TVZ2018}. For error estimates of further sparsity promoting optimal control problems with PDEs see for example \cite{CMR2018,KPV2014}.
The literature on error estimates for optimal control problems with controls in BV is rather limited. We are only aware of \cite{CKP1999,Florian2017}.
Error estimates and numerical analysis for inverse problems involving BV-functions are studied in \cite{Bartels2012,BM2017}.
Related discussion of ODE-constrained control problems involving discontinuous functions and their numerical analysis can be found in, e.g., \cite{C2006,CCLZ2011,D1996,DHV2000,V1997,V2005,ABLG2013,AFS2018}.

The main difficulty in deriving error estimates for the above problem is given by the fact that it lacks certain coercivity properties that are usually employed to obtain error estimates for the controls, for instance by suitably testing the first order necessary optimality conditions. Hence, only error estimates for the state and the adjoint state can be proven in a rather direct manner; these are, however, suboptimal. To obtain an error estimate for the control and also to improve the error estimates for state and adjoint state, we make use of a structural assumption on the Lagrange multiplier $\bar\Phi$ arising from the convex subdifferential of the term $\|q'\|_{\mathcal{M}(\Omega)}$.
Specifically, we assume that $\bar \Phi$, which is a $C^2$ function in $\bar\Omega$, has only finitely many global extreme points and that it exhibits quadratic growth near those points (i.e., $\bar\Phi''\neq 0$ near those points; see \Cref{A2} and \Cref{A3}). 
Since the jump set of the optimal control is contained in the set of global extreme points of $\bar\Phi$, see \Cref{cor:supportcondition}, this assumption implies that the optimal control admits only finitely many jumps, which is a rather typical situation in practice. 
In addition, it ensures that the adjoint state has nonvanishing first derivative near the global extreme points of $\bar\Phi$,
which is closely related to assumptions used to derive error estimates for bang-bang control problems, see, e.g., \cite{BPV2019,CWW2018,DeckelnickHinze2012}.

Starting from possibly suboptimal error estimates for the state and adjoint state and incorporating the structural assumption, we are able to derive an error estimate for the controls in $L^1$ for both variational control discretization, where the order of the error is $\mathcal{O}(h^2)$, and piecewise constant control discretization, where we obtain $\mathcal{O}(h)$. Moreover, we provide numerical experiments which indicate that the established error estimates are optimal. To further substantiate the use of the $L^1$-norm in the error estimates for the control, we include numerical results for the order of convergence of the controls with respect to the $L^2$-norm. These results clearly show that in both discretization schemes the order of convergence in $L^1$ is higher than the one in $L^2$. 

Let us stress that the essential structural assumption on $\bar\Phi$ cannot be transferred to settings in which the control domain is of dimension greater than one.
This is due to the fact that in such settings the Lagrange multiplier $\bar\Phi$ does not characterize the jump set of the optimal control.
While this implies that the \emph{control} domain is limited to an interval, this is not the case for the domain of the \emph{state}. 
We expect that the analysis presented in this paper can be extended
to problems where the state lives on a domain of dimension larger than one.

This paper is structured as follows. In \Cref{sec_contprob} we provide the precise problem setting and discuss existence of optimal solutions as well as first order optimality conditions for the continuous problem. \Cref{sec_discrete} is concerned with the same aspects, but for the two discretization schemes. In \Cref{sec_error} we derive both the basic and the improved error estimates, which is why this section also contains the structural assumption. The numerical experiments are presented in \Cref{sec_numerics}.


\section{The continuous problem}\label{sec_contprob}

We will consider the following model problem in the one-dimensional spatial domain $\Omega:=(0,1)$.
Given the parameter $\alpha>0$, a desired state $u_d\in L^\infty(\Omega)$,  and functions $a\in C^{0,1}(\bar\Omega)$ and $d_0\in \Linf$ satisfying $a(x)\geq \nu>0$ with a constant $\nu>0$ for all $x\in\bar\Omega$ and $d_0(x)\geq 0$ for a.e. $x\in\Omega$, we are looking for a control $q\in Q:=\BV$ and an associated state $u\in V:=\Hoz$ solving the optimal control problem
\begin{equation*}
\min\limits_{(u,q)\in V\times Q} \,\underbrace{\frac12\Norm{u-u_d}{\LL}^2 + \alpha\Norm{q'}{\M}}_{=:J(u,q)}\quad
\text{s.t.} \quad 
\cala(u,w) = (q,w)_{\LL} \enspace \forall \, w\in V,
\end{equation*}
where the bilinear form $\cala$ is given by 
\begin{equation*}
	\cala\colon V \times V\rightarrow\R,\qquad
	\cala(v,w) := (a v',w')_{\LL}+(d_0 v,w)_{\LL}.
\end{equation*}

\subsection{The state equation}

Recall from, e.g., \cite{ambrosio,Giusti,Ziemer} that the space $\BV$ is given by those functions $v\in\Lone$ for which the distributional derivative $v'$ is a Radon measure, i.e.,
\begin{equation*}
\BV = \left\{v\in\Lone\colon \Norm{v'}{\M} < \infty\right\},
\end{equation*}
where $\M$ denotes the space of Radon measures.
The space $\BV$ is a Banach space if equipped with the norm
\begin{equation*}
\Norm{v}{\BV} := \Norm{v}{\Lone} + \Norm{v'}{\M},
\end{equation*}
see, e.g., \cite[Thm.~10.1.1]{Attouch}. Moreover, 
$\BV$ embeds continuously into $L^p(\Omega)$ for $p\in[1,\infty]$ and compactly into $L^p(\Omega)$ for $p\in [1,\infty)$, see, e.g., \cite[Cor.~3.49 together with Prop.~3.21]{ambrosio}. 
As $\BV$ embeds into $\LL$ we note that for every $q\in\BV$ the Lax-Milgram theorem
readily guarantees existence of a unique associated state $u=u(q)\in V$.
Thus, the use of the solution or control-to-state operator
\begin{equation*}
S\colon Q \subset V^* \rightarrow V
\end{equation*}
is justified. We note in passing that $S\colon V^*\rightarrow V$ is a self-adjoint isomorphism.
In fact, because we are working in dimension one, the following strong regularity result can be proven by standard arguments.

\begin{lemma}\label{lem_aprioriS}
Let $p\in (1,\infty]$. For all $v\in L^p(\Omega)$ there holds $S v \in W^{2,p}(\Omega)\cap V$, and the estimate
\begin{equation*}
	\Norm{S v}{W^{2,p}(\Omega)} \leq C \Norm{v}{L^p(\Omega)}
\end{equation*}
is satisfied, where the constant $C>0$ is independent of $v$ and $p$.
\end{lemma}

Introducing the reduced objective $j\colon Q\rightarrow\R$, $j(q):=J(S(q),q)$, we can now analyze the reduced version of the original problem, given by
\begin{equation}\label{prob:BV}\tag{P}
 \min_{q\in Q}j(q).
\end{equation}
We will demonstrate that \eqref{prob:BV} admits a unique solution, characterize this solution by means of optimality conditions, and draw some conclusions from the optimality conditions regarding the structure of the optimal solution. Due to convexity we need not distinguish between local and global solutions, and first order necessary conditions are also sufficient.

\subsection{Existence of optimal controls}
\begin{theorem} \label{thm:existenceofsol}
	Problem~\eqref{prob:BV} admits a unique optimal control  $\bar q \in Q$ with associated optimal state $\bar u\in V$.
\end{theorem}

\begin{proof}
The injectivity of $S$ implies that $j$ is strictly convex, so \eqref{prob:BV} has at most one solution.
To establish existence of $\bar q$, let us consider a minimizing sequence $(q_n)_{n\in\N}$ of $j$ with $j(q_n)\leq j(0)$ for all $n\in\N$.
Our goal is to bound the BV-norm of that sequence. Since there holds
\begin{equation} \label{eq:existenceproofm1}
\alpha \lVert q_n' \rVert_{\M} \leq j(q_n) \leq j(0),
\end{equation}
it only remains to establish that $(\norm{q_n}{\Lone})_{n\in\N}$ is bounded. From \cite[Thm. 3.44]{ambrosio} it follows that
\begin{equation} \label{eq:existenceproof0}
\lVert q_n - \hat q_n \rVert_{L^1(\Omega)} \leq C_{iso} \lVert q_n' \rVert_{\M} \leq \frac{C_{iso}j(0)}{\alpha},
\end{equation}
where $\hat q_n := \frac{1}{|\Omega|} \int_\Omega q_n \,dx$ and $C_{iso}$ depends only on $\Omega.$ Estimate \eqref{eq:existenceproof0} implies via the inverse triangle inequality that for all $n\in\N$ there holds
\begin{equation} \label{eq:existenceproof1}
\lVert q_n \rVert_{L^1(\Omega)} \leq \frac{C_{iso}j(0)}{\alpha} + |\hat q_n|,
\end{equation}
where we have used that $\lvert\Omega\rvert=1$.
Moreover, we have
\begin{equation*}
\begin{split}
|\hat q_n| \lVert S1 \rVert_{\LL}  
& \leq \lVert S (\hat q_n - q_n) \rVert_{\LL} +\lVert S q_n \rVert_{\LL}\\
& \leq \lVert S \rVert_{\Lin(V^*, \LL)}  \Norm{\hat q_n - q_n}{V^*} +\Norm{Sq_n}{\LL}.
\end{split}
\end{equation*}
Making use of the embedding $\Lone\hookrightarrow V^*$ with constant $C_{emb}$ we infer that the first term on the right-hand side can be bounded using \eqref{eq:existenceproof0}, and the second term can be bounded by \eqref{eq:existenceproofm1}. Together, this yields
\begin{align*}
|\hat q_n| \lVert S1 \rVert_{\LL} & \leq \frac{C_{iso}C_{emb} j(0)}{\alpha} \lVert S \rVert_{\Lin(V^*, \LL)} + \lVert Sq_n - u_d \rVert_{\LL}+\lVert u_d \rVert_{\LL}\\
& \leq \frac{C_{iso}C_{emb} j(0)}{\alpha} \lVert S \rVert_{\Lin(V^*, \LL)} + \sqrt{2 j(0) }+ \lVert u_d \rVert_{\LL}.
\end{align*}
This and \eqref{eq:existenceproof1} imply
\begin{equation*} 
\begin{split}
\lVert q_n \rVert_{L^1(\Omega)} & \leq \frac{C_{iso}j(0)}{\alpha} \\
& ~ + \lVert S1 \rVert_{\LL}^{-1} \left( \frac{C_{iso}C_{emb} j(0)}{\alpha} \lVert S \rVert_{\Lin(V^*, \LL)} + 2\sqrt{2 j(0) } \right),
\end{split}
\end{equation*}
where we have used that $S1\neq 0$ and that $\lVert u_d \rVert_{\LL} \leq \sqrt{2j(0)}$. 
In view of \eqref{eq:existenceproofm1} we have thus found for all $n\in\N$ that
\begin{equation} \label{eq:existenceproof3}
\begin{split}
\lVert q_n \rVert_{BV(\Omega)} & \leq \frac{(C_{iso}+1)j(0)}{\alpha} \\
& ~ + \lVert S1 \rVert_{\LL}^{-1} \left( \frac{C_{iso}C_{emb} j(0)}{\alpha} \lVert S \rVert_{\Lin(V^*, \LL)} + 2 \sqrt{2 j(0) } \right).
\end{split}
\end{equation}
Since $\BV$ is compactly embedded in $\Lone$, there is a subsequence $(q_{n_k})_{k\in\N}$ of $(q_n)$ and a $\bar q\in\Lone$ such that $q_{n_k}\to\bar q$ in $\Lone$ for $k\to\infty$.
By continuity of the mapping $\LL \ni q\mapsto\frac{1}{2}\norm{Sq -u_d}{\LL}^2$ and lower semicontinuity of $q\mapsto \norm{q'}{\M}$ with respect to the $L^1(\Omega)$ topology, cf. \cite[Thm.~5.2.1]{Ziemer}, we deduce that $j(\bar q) = \inf_{q\in Q} j(q)$. 
\end{proof}

\subsection{Optimality conditions}
Next, we provide necessary and sufficient optimality conditions for the optimal solution.

\begin{theorem}\label{thm_optcond}
	The control $\bar q\in Q$ with associated state $\bar u\in V$ is optimal for Problem~\eqref{prob:BV} if and only if
	there exists a unique adjoint state $\bar z\in W^{2,\infty}(\Omega) \cap V$ such that $(\bar u,\bar q,\bar z)$ and 
	the $W^{3,\infty}(\Omega)$ function $\bar\Phi\colon[0,1]\rightarrow\R$, $\bar\Phi(x):=\int_{0}^{x} \bar z(s)\,ds$ satisfy
	$\bar\Phi(1)=0$ as well as
	\begin{equation*}
	\begin{split}
	&\int_\Omega \bar\Phi \,d\bar q'=\alpha\Norm{\bar q'}{\M}\qquad\text{and}\qquad
	\norm{\bar\Phi}{\infty}\leq\alpha,\\\\
	&
		\begin{aligned}
			\cala(\bar u,w) & = (\bar q,w)_{\LL} &\forall\, w\in V,\\
			\cala(w,\bar z) & = (w,\bar u-u_d)_{\LL} &\forall\, w\in V,\\
		\end{aligned}
	\end{split}
	\end{equation*}

and

\begin{equation*}
- \left( \bar z, q-\bar q \right)_{\LL} \leq \alpha \left[\Norm{q'}{\M} - \Norm{ \bar q'}{\M}\right] \quad \forall q\in Q.
\end{equation*}
\end{theorem}

\begin{proof}
Using convex analysis, e.g. \cite{Peypouquet}, the optimality of $\bar q$ is equivalent to
\begin{equation*}
0 \in \partial j(\bar q),
\end{equation*}
where $\partial j(\bar q)$ denotes the subdifferential of $j$ at the point $\bar q$. By the chain rule and the sum rule, e.g. \cite[Proposition 3.28]{Peypouquet} and \cite[Thm. 3.30]{Peypouquet}, this is equivalent to
\begin{align} \label{eq:proof:optimalitycond0}
- S^*(S\bar q - u_d) \in \partial \left( \alpha \Norm{\bar q'}{\M} \right).
\end{align}
Note that the sum rule is applicable since both summands of $j$ are continuous on $Q$.
Defining $\bar z := S^*(S\bar q - u_d)$ and recalling $\bar u=S\bar q$ we obtain 
\begin{equation*}
\begin{aligned}
			\cala(\bar u,w) & = (\bar q,w)_{\LL} &\forall\, w\in V,\\
			\cala(w,\bar z) & = (w,\bar u-u_d)_{\LL} &\forall\, w\in V.\\
		\end{aligned}
\end{equation*}
In particular, the asserted regularity of $\bar z$ follows from \Cref{lem_aprioriS}, which in turn
implies $\bar\Phi\in W^{3,\infty}(\Omega)$.
Furthermore, the definition of the subdifferential implies that \eqref{eq:proof:optimalitycond0} can be equivalently expressed as
\begin{equation*}
-\left( \bar z, q - \bar q \right)_{\LL} \leq \alpha \left[ \lVert q' \rVert_{\M} - \lVert \bar q' \rVert_{\M} \right] \quad \forall q\in Q.
\end{equation*}
Testing with $q = 2 \bar q$, $q = 0$ and $q = \tilde q+\bar q$ for any $\tilde q\in Q$ yields the equivalent system
\begin{equation} \label{eq:proof:optimalitycond1}
\begin{split}
- \left( \bar z, \bar q \right)_{\LL} & = \alpha\lVert \bar q' \rVert_{\M},\\
\lvert \left( \bar z, q \right)_{\LL} \rvert & \leq \alpha \lVert q' \rVert_{\M} \quad \forall q\in Q.
\end{split}
\end{equation}
Inserting $q=1$ into \eqref{eq:proof:optimalitycond1} supplies $\bar\Phi(1) = \int_\Omega \bar z\,ds = 0$.
By the definition of the distributional derivative of BV functions, \eqref{eq:proof:optimalitycond1} is equivalent to
\begin{equation} \label{eq:proof:optimalitycond2}
\begin{split}
\int_\Omega \bar\Phi \,d\bar q' & = \alpha\lVert \bar q' \rVert_{\M},\\
\Abs{\int_\Omega \bar\Phi \,d q'} & \leq \alpha \lVert q' \rVert_{\M} 
\quad \forall q\in Q.
\end{split}
\end{equation}
For $x\in \Omega$ let $q:= 1_{(x,1)} \in Q$ be the characteristic function of the interval $(x,1)$. We have $q' = \delta_x$ and hence \eqref{eq:proof:optimalitycond2} yields $\abs{\bar\Phi(x)}\leq \alpha$. 
\end{proof}

\paragraph{\textbf{Structural conclusions}}
With the optimality conditions of Theorem \ref{thm_optcond} at hand, we can now derive helpful structural properties that hold without additional assumptions.

%

\begin{corollary} \label{cor:supportcondition}
	If $\bar q$ is optimal for \eqref{prob:BV}, then there hold 
	\begin{equation*}
	\begin{aligned}
	&\supp(\bar q'_+)\subset\left\{x\in\Omega\colon\,\bar\Phi(x)=\alpha\right\},&\\
	&\supp(\bar q'_-)\subset\left\{x\in\Omega\colon\,\bar\Phi(x)=-\alpha\right\},&
	\end{aligned}
	\end{equation*}
	where $\bar q'_+$ and $\bar q'_-$ denote the positive and the negative part of the Jordan decomposition of the measure $\bar q'$. Moreover, we have
	\begin{equation}\label{eq:jumpsatrootsofz}
	\supp(\bar q')
	\subset\left\{x\in\Omega\colon\,\Abs{\bar\Phi(x)}=\alpha\right\}
	\subset\bigl\{x\in\Omega\colon\,\bar z(x)=0\bigr\}.
	\end{equation}
\end{corollary}

\begin{proof}
Let $\hat x\in \Omega$ with $\bar \Phi(\hat x) < \alpha$. By the continuity of $\bar\Phi$ there is an open neighborhood $U\subset\Omega$ of $\hat x$ and $\delta>0$ such that $\bar\Phi \leq \alpha-\delta$ on $U$. Then we have
\begin{align*}
\alpha \lVert \bar q' \rVert_{\M} & = \int_\Omega \bar\Phi ~d\bar q'_+ - \int_\Omega \bar\Phi~d\bar q'_- \leq \int_{\Omega \setminus U} \alpha ~d\bar q'_+ + \int_U (\alpha-\delta) ~d\bar q'_+ + \int_\Omega\alpha ~d\bar q'_-\\
& = \alpha \lVert \bar q' \rVert_{\M} - \delta \bar q'_+(U).
\end{align*}
Thus $\bar q'_+(U) = 0$ and $\hat x\not\in \supp(\bar q'_+)$. The claim for $\bar q'_-$ follows analogously. The first inclusion in \eqref{eq:jumpsatrootsofz} follows from
	\begin{equation*}
	\supp(\bar q')=\supp(\bar q'_+)\cup\supp(\bar q'_-)\subset\left\{x\in\Omega\colon\,\Abs{\bar\Phi(x)}=\alpha\right\}.
	\end{equation*}
	\Cref{thm_optcond} implies that every $x$ with $\abs{\bar\Phi(x)}=\alpha$ is either a global maximum or minimum of the $C^1$ function $\bar\Phi$ and hence satisfies $0 = \bar\Phi'(x) = \bar z(x)$, establishing the second inclusion in \eqref{eq:jumpsatrootsofz}.	
\end{proof}


\section{Finite element discretization}\label{sec_discrete}
For the discretization of \eqref{prob:BV} we divide $\bar\Omega=[0,1]$ into $1 < l$ subintervals $T_i=(x_{i-1},{x_i})$ of size $h_i$ defined by the spatial nodes
\begin{equation*}
 	0=x_0<x_1<\ldots<x_l=1, \qquad \mathcal{N}_h := \bigl\{x_0,x_1,\dots,x_l\bigr\}.
\end{equation*}
We obtain $\Omega=\bigcup\limits_{1\leq i\leq l} T_i$ and set ${\mathcal{T}}_h :=\bigcup\limits_{1\leq i\leq l} \{T_i\}$,
where $h:=\max_{1\leq i\leq l}h_i$ denotes the mesh width.

\subsection{Discretization of the state equation}\label{subsec_discrete_state}

To discretize the state equation we use linear finite elements, i.e., the discrete state space $V_h$ is given by
\begin{equation*}
V_h := \left\lbrace v_h\in V \cap C_0(\bar \Omega): \; v_h|_{T} \text{ is linear for all }T\in\mathcal{T}_h \right\rbrace,
\end{equation*}
where $C_0(\bar \Omega)$ denotes the continuous functions on $\bar\Omega$ that vanish on $\partial\Omega$.

For further reference we recall that the \emph{Ritz projection} associated to the bilinear form $\cala$, denoted $R_h\colon V \to V_h$, satisfies
\begin{equation*}
\cala(R_hv,w_h)=\cala(v,w_h)\quad\forall\, w_h\in V_h.
\end{equation*}
It is well known that for each $v\in V$ this variational equality has a unique solution. 
Moreover, the discrete solution operator is denoted by $S_h\colon V^* \rightarrow V_h$ and satisfies, with $u_h := S_hv$,
\begin{equation*}
\cala(u_h,w_h) = (v,w_h)_{\LL} \quad\forall\, w_h\in V_h.\\
\end{equation*}
Since these identities, in fact, uniquely determine $R_h$ and $S_h$, it follows that $S_h = R_hS$ on $V^*$. 

Concerning the approximation quality of $S_h$ we cite the following well-known results.

\begin{lemma}\label{lem_aprioridiffSSh}
	There exist $C>0$ and $h_0>0$ such that for every $h\in(0,h_0]$ and all $v\in\LL$ there hold
	\begin{equation*}
	\Norm{S v-S_h v}{\LL} \leq C h^2\norm{v}{\LL}\qquad\text{and}\qquad
	\Norm{S v-S_h v}{V} \leq C h\norm{v}{\LL}.
	\end{equation*}
\end{lemma} 
	
\begin{proof}
	Cf., e.g., \cite[Section 3.2]{ErnGuermond2004}.
\end{proof}
	
\begin{lemma}\label{lem_ritz}
	There exist $C>0$ and $h_0>0$ such that for every $h\in(0,h_0]$ and all $v\in L^\infty(\Omega)$ there holds
	\begin{equation*}
	\lVert Sv - S_hv\rVert_{L^\infty(\Omega)} \leq C h^2 \lVert v \rVert_{L^\infty(\Omega)}.
	\end{equation*}
\end{lemma}
	
\begin{proof}
	This is the main theorem of \cite{Wheeler1973}, keeping the regularity from \Cref{lem_aprioriS} in mind.
\end{proof}
	
The next lemma shows that $S_h$ is stable from $\LL$ to $\Woneinf$. 
\begin{lemma}\label{lem_aprioriWoneinf}
	There exist $C>0$ and $h_0>0$ such that for every $h\in(0,h_0]$ and all $v\in\LL$ there holds
	\begin{equation*}
	\Norm{S_h v}{W^{1,\infty}(\Omega)} \leq C \lVert Sv\rVert_{W^{1,\infty}(\Omega)} \leq C \Norm{v}{\LL}.
	\end{equation*}
\end{lemma}

\begin{proof}
This is a consequence of the stability result from \cite[Thm. 8.1.11]{brenner}, the embedding $H^2(\Omega)\hookrightarrow W^{1,\infty}(\Omega)$, and \Cref{lem_aprioriS}.
\end{proof}

\begin{lemma} \label{lem:W1inftystabofritz}
Let $w\in H^2(\Omega)\cap V$ and $R_hw$ its Ritz projection. Then there are $C,h_0>0$ such that for each $h\in (0,h_0]$ we have
\begin{equation*}
\lVert (R_hw-w)^\prime\rVert_{L^\infty(\Omega)} \leq Ch^{\frac{1}{2}} \lVert w \rVert_{H^2(\Omega)}.
\end{equation*}
If $w\in W^{2,\infty}(\Omega)\cap V$ we even have
\begin{equation*}
\lVert (R_hw-w)^\prime\rVert_{L^\infty(\Omega)} \leq Ch \lVert w \rVert_{W^{2,\infty}(\Omega)}.
\end{equation*}
In both cases, the constant $C>0$ is independent of $w$ and $h$.
\end{lemma}

\begin{proof}
\Cref{lem_aprioriWoneinf} implies that the Ritz projection is stable in $W^{1,\infty}(\Omega)$ and thus
\begin{align*}
\lVert R_h w-w \rVert_{W^{1,\infty}(\Omega)} & \leq \lVert R_h ( w- I_h w) \rVert_{W^{1,\infty}(\Omega)} + \lVert I_h w- w \rVert_{W^{1,\infty}(\Omega)} \\
& \leq C\lVert I_h w- w \rVert_{W^{1,\infty}(\Omega)}.
\end{align*}
Here, $I_hw$ is the usual nodal interpolant of $w$. The two estimates now follow from \cite[Thm. 4.4.20]{brenner}.
\end{proof}

\subsection{Variational control discretization}\label{subsec_vd}
In this section we discuss the variational discretization of problem~\eqref{prob:BV}, in which the controls are not discretized explicitly. 
We show that the resulting semi-discrete problem admits a unique solution, characterize this solution by means of optimality conditions, 
and draw conclusions from the optimality conditions regarding the structure of the optimal solution.

The variationally discretized version of \eqref{prob:BV} is given by
\begin{equation*}
\begin{split}
&\min_{(u_h,q)\in V_h\times Q} 
\frac12\Norm{u_h-u_d}{\LL}^2 + \alpha\Norm{q'}{\M}\\
&\qquad\text{s.t.}\qquad
\cala(u_h,w_h) = (q,w_h)_{\LL} \enspace \forall \, w_h\in V_h.
\end{split}
\end{equation*}
Defining $j_h\colon Q\rightarrow\R$ by $j_h(q):=J(S_h(q),q)$, its reduced formulation reads
\begin{equation}\label{prob_BVh}\tag{P$_{\text{vd}}$}
 \min_{q\in Q} j_h(q).
\end{equation}

\Cref{thm:existenceofsol} has the following discrete counterpart.

\begin{theorem} \label{thm:existenceofsolh}
Problem~\eqref{prob_BVh} admits a unique optimal control $\bar q_h\in Q$ with associated optimal state $\bar u_h\in V_h$.
There exist $C>0$ and $h_0>0$ such that for all $h\in (0,h_0]$ the controls satisfy $\lVert \bar q_h \rVert_{\BV} \leq C$.
\end{theorem}

\begin{proof}
The proof of \Cref{thm:existenceofsol} can be used verbatim as there holds $S_h1\neq 0$. It remains to establish the estimate for the controls. As in the proof of \Cref{thm:existenceofsol} we can derive \eqref{eq:existenceproof3}
for $\bar q_h$ instead of $\bar q$.
Passing to the limit in this version of 
\eqref{eq:existenceproof3} yields 
\begin{equation} \label{eq:proof:barqhbound}
\begin{split}
\lVert \bar q_h \rVert_{BV(\Omega)} & \leq \frac{(C_{iso}+1)j_h(0)}{\alpha} \\
& ~ + \lVert S_h1 \rVert_{\LL}^{-1} \Biggl[ \frac{C_{iso}C_{emb} j_h(0)}{\alpha} \lVert S_h \rVert_{\Lin{(V^*, \LL)}} + \sqrt{8 j_h(0) } \Biggr].
\end{split}
\end{equation}
From \Cref{lem_aprioridiffSSh} we obtain
\begin{equation} \label{eq:proof:Sh1bound}
\Norm{S1}{\LL}-\Norm{S_h 1}{\LL}\leq 
\Norm{S1 - S_h 1}{\LL} \leq Ch^2 \lVert 1 \rVert_{\LL} 
\leq \frac12 \lVert S1 \rVert_{\LL} 
\end{equation}
for $h$ sufficiently small, thus $\lVert S_h 1\rVert_{\LL} \geq \frac12\lVert S1\rVert_{\LL}$. Furthermore, $S_h = R_hS$ on $V^*$ and the $H^1(\Omega)$-stability of the Ritz projection imply 
\begin{equation} \label{eq:proof:Shboundreflater}
\lVert S_h \rVert_{\Lin{(V^*, \LL)}} \leq C \lVert S \rVert_{\Lin{(V^*, V)}}.
\end{equation}
Inequalities \eqref{eq:proof:Sh1bound} and \eqref{eq:proof:Shboundreflater} in conjunction with \eqref{eq:proof:barqhbound} and $j_h(0) = j(0)$ yield the desired boundedness of $\lVert \bar q_h\rVert_{\BV}$ independent of $h$.
\end{proof}

We point out that the control space $Q$ is not discretized, hence the optimal control $\bar q_h$ belongs to $\BV$. We prefer the notation $\bar q_h$ nonetheless, 
because the variationally discretized problem depends on $h$.

We collect without proof optimality conditions and structural properties analogous to the continuous setting.
\begin{theorem}\label{thm_optcondh}
	The control $\bar q_h\in Q$ with associated state $\bar u_h\in V_h$ is optimal for Problem~\eqref{prob_BVh} if and only if
	there exists a unique adjoint state $\bar z_h\in V_h$ such that $(\bar u_h,\bar q_h,\bar z_h)$ and 
	the $C^1$ function $\bar\Phi_h\colon[0,1]\rightarrow\R$, $\bar\Phi_h(x):=\int_{0}^{x} \bar z_h(s)\,ds$ satisfy
	$\bar\Phi_h(1)=0$ as well as
	\begin{equation*}
		\begin{split}
		\int_{\Omega} \bar\Phi_h \,d\bar q_h'=\alpha\Norm{\bar q_h'}{\M}\qquad\text{and}\qquad \lVert \bar\Phi_h\rVert_{L^\infty(\Omega)} \leq \alpha, \\\\
		\begin{aligned}
		\cala(\bar u_h,w_h) & = (\bar q_h,w_h)_{\LL} &\forall\, w_h\in V_h,\\
		\cala(w_h,\bar z_h) & = (w_h,\bar u_h-u_d)_{\LL} &\forall\, w_h\in V_h,\\
		\end{aligned}
		\end{split}
	\end{equation*}	

and

\begin{equation*}
- \left( \bar z_h, q-\bar q_h \right)_{\LL} \leq \alpha \left[\Norm{q'}{\M} - \Norm{\bar q_h'}{\M}\right] \quad \forall q\in Q. 
\end{equation*}
\end{theorem}


\bigskip

\begin{corollary}\label{cor:discretesupportcondition}
	If $\bar q_h$ is optimal for \eqref{prob_BVh}, then there hold 
	\begin{equation*}
	\begin{aligned}
		&\supp((\bar q_h')_+)\subset\left\{x\in\Omega\colon\,\bar\Phi_h(x)=\alpha\right\},&\\
		&\supp((\bar q_h')_-)\subset\left\{x\in\Omega\colon\,\bar\Phi_h(x)=-\alpha\right\},&
	\end{aligned}
	\end{equation*}
	where $(\bar q_h')_+$ and $(\bar q_h')_-$ denote the positive and the negative part of the Jordan decomposition of the measure $\bar q_h'$. Moreover, we have
	\begin{equation*}
		\supp(\bar q_h')
		\subset\left\{x\in\Omega\colon\,\Abs{\bar\Phi_h(x)}=\alpha\right\}
		\subset\bigl\{x\in\Omega\colon\,\bar z_h(x)=0\bigr\}.
	\end{equation*}
\end{corollary}

\subsection{Piecewise constant control discretization}\label{subsec_fd}

In this section we present a discretization for \eqref{prob:BV} in which the controls $q_h$ are piecewise constant. We denote the space of piecewise constant functions on $\mathcal{T}_h$ by 
\begin{equation*}
Q_h := \left\lbrace q_h \in \BV: \; q_h\vert_T = \operatorname{const.} ~ \text{ for all } T\in\mathcal{T}_h \right\rbrace.
\end{equation*}
Now the discretization of \eqref{prob:BV} is given by 
\begin{equation*}
\begin{split}
& \min_{(u_h,q_h)\in V_h\times Q_h} \frac12\Norm{u_h-u_{d}}{\LL}^2 + \alpha\Norm{q_h'}{\M} \\
& \qquad\;\,\text{s.t.}\qquad
\cala(u_h,w_h) = (q_h,w_h)_{\LL} \enspace\forall \, w_h\in V_h.
\end{split}
\end{equation*}
With $j_h(q_h) := J(S_h(q_h),q_h)$ its reduced formulation reads
\begin{equation}\label{prob_BVch}\tag{P$_{\text{cd}}$}
\min_{q_h\in Q_h} j_h(q_h).
\end{equation}
Note that in contrast to \eqref{prob_BVh} the control $q_h$ is now discretized and has the form
\begin{equation}\label{eq:barqhpdarstellung}
q_h  = a_h+\sum_{j=1}^{l-1} c_h^j 1_{(x_j,1)},
\qquad\text{hence}\qquad
q_h' = \sum_{j=1}^{l-1} c_h^j \delta_{x_j}
\end{equation}
for some $a_h, c_h^j \in\mathbb{R}$, $1\leq j\leq l-1$.

We now address existence of optimal solutions and optimality conditions for Problem~\eqref{prob_BVch}. 

\begin{theorem} \label{thm:existence_probch}
Problem~\eqref{prob_BVch} admits a unique optimal control $\hat q_h\in Q_h$ with associated optimal state $\hat u_h\in V_h$. There exist $C>0$ and $h_0>0$ such that for all $h\in (0,h_0]$ we have $\lVert \hat q_h \rVert_{BV(\Omega)} \leq C$.
\end{theorem}

\begin{proof}
	The proof is the same as for \Cref{thm:existenceofsolh}.
\end{proof}

\begin{theorem} \label{thm_optcondch}
	The control $\hat q_h\in Q_h$ with associated state $\hat u_h\in V_h$ is optimal for Problem~\eqref{prob_BVch} if and only if there exists a unique adjoint state $\hat z_h\in V_h$ such that $(\hat u_h,\hat q_h,\hat z_h)$ and the $C^1$ function $\hat\Phi_h\colon[0,1]\rightarrow\R$, $\hat\Phi_h(x):=\int_{0}^{x} \hat z_h(s)\,ds$ satisfy
	$\hat\Phi_h(1)=0$ as well as
	\begin{equation*}
	\begin{split}
	\int_{\Omega} \hat\Phi_h \,d\hat q_h'=\alpha\Norm{\hat q_h'}{\M}\qquad\text{and}\qquad
	\max_{0\leq j\leq l}\left|\hat\Phi_h(x_j)\right|\leq\alpha,\\\\
			\begin{aligned}
			\cala(\hat u_h,w_h) & = (\hat q_h,w_h)_{\LL} &\forall\, w_h\in V_h,\\
			\cala(w_h,\hat z_h) & = (w_h,\hat u_h-u_d)_{\LL} &\forall\, w_h\in V_h,\\
		\end{aligned}
	\end{split}
	\end{equation*}	
	
	and
	
	\begin{equation*}
	- \left( \hat z_h, q_h-\hat q_h \right)_{\LL} \leq \alpha \left[\Norm{q_h'}{\M} - \Norm{\hat q_h'}{\M}\right] \quad \forall q_h\in Q_h.
	\end{equation*}
\end{theorem}

\begin{proof}
	As in the proof of Theorem~\ref{thm_optcond} the optimality of $\hat q_h\in Q_h$ is equivalent to
	\begin{equation*}
	- \hat z_h := - S_h^* (S_h \hat q_h - u_d) \in\partial\left( \alpha \lVert \hat q_h' \rVert_{\M} \right). 
	\end{equation*}
	Also as in the proof of Theorem~\ref{thm_optcond}, in particular \eqref{eq:proof:optimalitycond1}, this is equivalent to
	\begin{equation} \label{eq:proof:VImithatzh}
	\begin{split}
	- \left( \hat z_h, \hat q_h \right)_{\LL} & = \alpha \lVert \hat q_h' \rVert_{\M},\\
	\bigl| \left( \hat z_h, q_h \right)_{\LL} \bigr| & \leq \alpha \lVert q_h' \rVert_{\M}  \quad \forall q_h \in Q_h.
	\end{split}
	\end{equation}
	It remains to establish the statements for $\hat\Phi_h$. Testing with $q_h := 1 \in Q_h$ in \eqref{eq:proof:VImithatzh} shows $\int_\Omega \hat z_h(s)~ds = 0$ and thus $\hat \Phi_h(1) = 0$. Moreover, \eqref{eq:proof:VImithatzh} can be expressed as
	\begin{equation} \label{eq:proof:referencedlater}
	\begin{split} 
	\int_{\Omega} \hat \Phi_h \, d\hat q'_h & = \alpha \lVert \hat q'_h \rVert_{\M},\\
	\left| \int_{\Omega} \hat\Phi_h \, d q'_h \right| & \leq \alpha \lVert q'_h \rVert_{\M} \quad \forall q_h\in Q_h.
	\end{split}
	\end{equation}
	Because $1_{(x_j,1)} \in Q_h$ and $(1_{(x_j,1)})' = \delta_{x_j}$ for $j=0,1\dots,l$, we infer from the inequality in \eqref{eq:proof:referencedlater} that
	\begin{equation*}
	|\hat\Phi_h(x_j)| = \left|\int_\Omega\hat \Phi_h \, d (1_{(x_j,1)})' \right| \leq \alpha \lVert \delta_{x_j} \rVert_{\M} = \alpha. \qedhere
	\end{equation*}
\end{proof}

\begin{remark}
	The information on $\hat{\Phi}_h$ in \Cref{thm_optcondch} concerns only the gridpoints. It is therefore not ensured (and in general not true) that $\norm{\hat{\Phi}_h}{\infty}\leq\alpha$. 
\end{remark}

\begin{corollary}\label{cor_Phihattouchesboundhat}
	If $\hat q_h\in Q_h$ is optimal for \eqref{prob_BVch}, 
	then there holds
	\begin{equation*}
	\begin{aligned}
	&\supp((\hat q_h')_+)\subset\left\{x_j\in \mathcal{N}_h \colon\,\hat\Phi_h(x_j)=\alpha\right\},&\\
	&\supp((\hat q_h')_-)\subset\left\{x_j\in \mathcal{N}_h \colon\,\hat\Phi_h(x_j)=-\alpha\right\},&
	\end{aligned}
	\end{equation*}
	where $(\hat q_h')_+$ and $(\hat q_h')_-$ denote the positive and the negative part of the Jordan decomposition of the measure $\hat q_h'$.
\end{corollary}

\begin{proof}
	Recall that
	\begin{equation*}
	\hat q_h  = a_h+\sum_{j=1}^{l-1} c_h^j 1_{(x_j,1)},
	\qquad\text{hence}\qquad
	\hat q_h' = \sum_{j=1}^{l-1} c_h^j \delta_{x_j}
	\end{equation*}
	for real numbers $a_h, c_h^1,c_h^2,\dots,c_h^{l-1}$.
	Let $x_{j^\ast}\in\supp((\hat q_h')_+)$ for some $j^\ast\in\{1,\ldots,l-1\}$. Note that this is equivalent to saying that $c_h^{j^\ast}>0$. Assume that $\hat\Phi_h(x_{j^\ast}) < \alpha$. By \eqref{eq:proof:referencedlater} we have that
	\begin{equation*}
	\alpha\lVert \hat q_h^\prime\rVert_{\M} = \int_\Omega \hat \Phi_h \, d\hat q_h^\prime = \sum_{i=1, i\neq j^\ast}^{l-1} c_h^i \hat\Phi_h(x_i) + c_h^{j^\ast} \hat \Phi_h(x_{j^\ast}).
	\end{equation*}
	By \Cref{thm_optcondch} we thus find
	\begin{equation*}
	\alpha\lVert \hat q_h^\prime\rVert_{\M} < \sum_{i=1, i\neq j^\ast}^{l-1} |c_h^i| \alpha + c_h^{j^\ast} \alpha =  \alpha \sum_{i=1}^{l-1} |c_h^i| = \alpha\lVert \hat q_h^\prime\rVert_{\M},
	\end{equation*}
	a contradiction that implies $\hat\Phi_h(x_{j^\ast}) = \alpha$ and hence the statement for $\supp((\hat q_h^\prime)_+)$. Analogously, we obtain the assertion for $\supp((\hat q_h')_-)$.
\end{proof}

\begin{remark}
	Note that at non-gridpoints, $|\hat{\Phi}_h|$ may assume larger values than $\alpha$. This implies that $x_{j^\ast}$ with $|\hat{\Phi}_h(x_{j^\ast})|=\alpha$ is not necessarily an extreme point of $\hat{\Phi}_h$. It is therefore not ensured that $\hat\Phi_h'(x_{j^\ast})= \hat z_h(x_{j^\ast}) =0$. This stands in stark contrast to both the continuous and the variationally discretized problems, where every point at which $|\bar\Phi|$, respectively, $|\bar\Phi_h|$ attains the value $\alpha$ is necessarily an extreme point and thus a root of $\bar z$, respectively, $\bar z_h$.
	However, if $|\hat{\Phi}_h(x_{j^\ast})|=\alpha$ for some $j^\ast\in\{1,\ldots,l-1\}$, then Rolle's theorem yields the existence of $\xi\in (x_{j^\ast-1},x_{j^\ast+1})$ with $\hat\Phi_h^\prime(\xi) = \hat z_h(\xi)=0$. That is, there is a root of $\hat z_h$ whose distance to $x_{j^\ast}$ is no more than $h$. This will suffice to prove error estimates of order $\mathcal{O}(h)$.
\end{remark}

For later use let us define an $L^2$-projection operator onto the space of piecewise constant functions and collect useful properties of this operator.
\begin{definition}
For $i=0,1,\dots,l-1$ we introduce 
\begin{align*}
\Pi_h  \colon \BV\rightarrow Q_h, \quad
 \Pi_h q|_{(x_i,x_{i+1})} := (x_{i+1}-x_i)^{-1} \int_{x_i}^{x_{i+1}} q(s)\,ds.
\end{align*}
\end{definition}

It is easy to check that for any $v_h\in Q_h$ and $q\in\BV$ we have
\begin{equation} \label{eq:orthogonality}
	\left( \Pi_h q - q, v_h \right)_{L^2(\Omega)} = 0.
\end{equation}

%

We have the following estimates.
\begin{lemma} \label{lem:Piherrorestimate} \label{lem:gradnormstabilityofPih}
For any $q\in\BV$ there hold
\begin{itemize}
\item $\lVert \Pi_h q - q \rVert_{L^1(\Omega)} \leq h \rVert q^\prime\rVert_{\M}$,
\item $\lVert (\Pi_h q)^\prime\rVert_{\M} \leq \lVert q^\prime\rVert_{\M}$,
\item $\lVert q - \Pi_h q \rVert_{L^\infty(\Omega)} \leq h \lVert q^\prime\rVert_{L^\infty(\Omega)}$ provided $q\in W^{1,\infty}(\Omega)$.
\end{itemize}
\end{lemma}

\begin{proof}
The first two estimates are taken from \cite[Proposition 16]{Florian2017}.

By Rademacher's theorem (e.g. \cite[Thm. 2.14]{ambrosio}) $q$ is Lipschitz continuous with Lipschitz constant $\lVert q^\prime\rVert_{L^\infty(\Omega)}$. 
Thus, a straightforward estimate shows for any $i=0,1,\dots,l-1$ and $x\in (x_i,x_{i+1})$
\begin{equation*}
q(x) - \frac{1}{x_{i+1}-x_i} \int_{x_i}^{x_{i+1}} q(s)\,ds \leq \frac{\lVert q^\prime\rVert_{L^\infty(\Omega)}}{x_{i+1}-x_i} \int_{x_i}^{x_{i+1}} |x-s|\,ds \leq (x_{i+1}-x_i) \lVert q^\prime\rVert_{L^\infty(\Omega)}.
\end{equation*}
The definition of $h$ yields the desired last inequality.
\end{proof}

\section{Finite element error estimates}\label{sec_error}
\subsection{Error estimates for variational control discretization}\label{subsec_error_vd}
\subsubsection{Basic error estimates for state and adjoint state}

We begin this section by proving a priori estimates for the errors in the optimal state and the adjoint state.

\begin{lemma}\label{lem_1steptoh2convforthestates}
	There exist $C,h_0>0$ such that for all $h\in(0,h_0]$ we have
	\begin{equation*}
	\lVert \bar u - \bar u_h \rVert_{\LL}^2 \leq C (h^4 - \left( R_h \bar z - \bar z, \bar q_h -\bar q \right)_{\LL}).
	\end{equation*}
\end{lemma}

\begin{proof}
	The optimality conditions for $\bar q$ and $\bar q_h$ from Theorems \ref{thm_optcond} and \ref{thm_optcondh} provide
	\begin{align*}
	-\left( \bar z, \bar q_h - \bar q \right)_{\LL} 
	\leq \alpha \lVert \bar q_h'\rVert_{\M} - \alpha \lVert \bar q' \rVert_{\M},\\
	\left( \bar z_h, \bar q_h - \bar q \right)_{\LL} 
	\leq \alpha \lVert \bar q' \rVert_{\M} - \alpha \lVert \bar q_h' \rVert_{\M}.
	\end{align*}
	Adding these two inequalities and inserting $R_h\bar z$ yields
	\begin{equation} \label{eq:proof:steptoh2convforthestates1}
	\left( \bar z_h - R_h \bar z, \bar q_h -\bar q \right)_{\LL} + \left( R_h \bar z - \bar z, \bar q_h -\bar q \right)_{\LL} =
	\left( \bar z_h-\bar z, \bar q_h -\bar q \right)_{\LL} \leq 0.
	\end{equation}
	We can rearrange the first term by first using the state equations, cf. Theorems \ref{thm_optcond} and \ref{thm_optcondh}, and then using the definition of the Ritz projection. This demonstrates 
	\begin{equation*}
	\left( \bar z_h - R_h \bar z, \bar q_h -\bar q \right)_{\LL} 
	= \cala( \bar u_h - \bar u, \bar z_h - R_h \bar z) = \cala( \bar u_h - R_h \bar u, \bar z_h - \bar z).
	\end{equation*}
	Invoking the definition of the adjoint equations, cf. 
	Theorems \ref{thm_optcond} and \ref{thm_optcondh}, and $R_hS = S_h$ this reads
	\begin{equation*}
	\begin{split}
	\left( \bar z_h - R_h \bar z, \bar q_h - \bar q\right)_{\LL}
	& =
	\left( \bar u_h - \bar u, \bar u_h - R_h\bar u\right)_{\LL}\\ 
	& = \lVert \bar u_h - \bar u \rVert_{\LL}^2 - \left( \bar u_h-\bar u, \bar u - R_h\bar u\right)_{\LL}.
	\end{split}
	\end{equation*}
	Inserting this into \eqref{eq:proof:steptoh2convforthestates1} yields
	\begin{equation*}
	\lVert \bar u_h - \bar u \rVert_{\LL}^2 \leq - \left( R_h \bar z - \bar z, \bar q_h -\bar q \right)_{\LL} + \left( \bar u_h-\bar u, \bar u - R_h\bar u\right)_{\LL}.
	\end{equation*}
	H\"older's inequality and Young's inequality supply
	\begin{equation*}
	\lVert \bar u_h - \bar u \rVert_{\LL}^2 \leq - \left( R_h \bar z - \bar z, \bar q_h -\bar q \right)_{\LL} + \frac{1}{2} \lVert  \bar u_h-\bar u \rVert_{\LL}^2 + \frac{1}{2} \lVert \bar u - R_h \bar u \rVert_{\LL}^2.
	\end{equation*}
	Recalling that $\lVert \bar u - R_h \bar u \rVert_{\LL}^2=\lVert S\bar q - S_h \bar q \rVert_{\LL}^2 \leq Ch^4$ by \Cref{lem_aprioridiffSSh}, the assertion follows after subtraction of $\frac{1}{2} \lVert \bar u_h-\bar u \rVert_{\LL}^2$.
\end{proof}

\begin{lemma}\label{lem_aprioridiffuuh}
	There exist $C,h_0>0$ such that for all $h\in(0,h_0]$ we have
	\begin{equation*}
	\Norm{\bar u-\bar u_h}{\LL}\leq Ch.
	\end{equation*}
\end{lemma}

\begin{proof}
By Lemma~\ref{lem_1steptoh2convforthestates} we have that
\begin{equation*}
	\lVert \bar u - \bar u_h \rVert_{\LL}^2 \leq C(h^4 +  \lVert R_h \bar z - \bar z\rVert_{\LL} \lVert \bar q_h - \bar q\rVert_{\LL}).
\end{equation*}
By Lemma~\ref{lem_aprioridiffSSh} and \Cref{thm:existence_probch} the first term is of order $Ch^2$. Taking the root yields the desired estimate.
\end{proof}

We readily deduce an error estimate for the adjoint state.
\begin{lemma} \label{lem_aprioridiffzzh} \label{lem_aprioridiffzzhgrad}
	There exist $C,h_0>0$ such that for all $h\in (0,h_0]$ we have
	\begin{equation*}
	\lVert \bar z_h - \bar z\rVert_{W^{1,\infty}(\Omega)}\le Ch.
	\end{equation*}
\end{lemma}

\begin{proof}
	We have
	\begin{equation*}
	\begin{split}
	\lVert \bar z_h - \bar z\rVert_{W^{1,\infty}(\Omega)}
	& =\lVert S_h^*(\bar u_h -u_d)-S^*(\bar u-u_d) \rVert_{W^{1,\infty}(\Omega)}\\
	& \leq\lVert S_h^*(\bar u_h -u_d)-S^*(\bar u_h-u_d) \rVert_{W^{1,\infty}(\Omega)}+\lVert S^*(\bar u_h - \bar u) \rVert_{W^{1,\infty}(\Omega)}.
	\end{split}
	\end{equation*}
 	\Cref{lem_ritz} together with \Cref{lem:W1inftystabofritz} and
 	\Cref{lem_aprioriS} show that 
 	\begin{equation*}
 	\lVert S_h^*(\bar u_h -u_d)-S^*(\bar u_h-u_d) \rVert_{W^{1,\infty}(\Omega)} \leq C(h^2+h) \lVert \bar u_h - u_d \rVert_{L^\infty(\Omega)}\leq Ch,
 	\end{equation*}
 	where we used that $u_d\in L^\infty(\Omega)$ and that $\lVert \bar u_h \rVert_{L^\infty(\Omega)} \leq C$, the latter being a consequence of \Cref{lem_ritz}.
	Moreover, by means of $H^2(\Omega)\hookrightarrow W^{1,\infty}(\Omega)$ and \Cref{lem_aprioriS}	we obtain
	\begin{equation*}
	\lVert S^*(\bar u_h - \bar u) \rVert_{W^{1,\infty}(\Omega)}
	\leq C \lVert \bar u_h - \bar u \rVert_{L^2(\Omega)}\leq Ch,
	\end{equation*}
	where the last inequality is due to Lemma~\ref{lem_aprioridiffuuh}.
\end{proof}

\subsubsection{Improved error estimates under structural assumptions}

We now improve the $\LL$ convergence order for the state to $\mathcal{O}(h^2)$ and deduce from this
that the controls have $\Lone$ convergence order $\mathcal{O}(h^2)$, and that
the adjoint state has $\Linf$ convergence order $\mathcal{O}(h^2)$. 
To achieve this, we work with a structural assumption: We consider situations where the continuous optimal control admits finitely many jumps. More precisely, we assume that the number of minima and maxima of the function $\bar\Phi$ is finite. This number bounds the number of jumps of the optimal control. Since these maxima and minima are in fact roots of the continuous adjoint state, regularity and convergence results for the discrete adjoint state allow to prove that the discrete problem admits a similar structure.
In the following we will frequently use the regularity $\bar z\in W^{2,\infty}(\Omega)$ from \Cref{thm_optcond}.

The essential structural assumption reads as follows.

\begin{assumption}\label{A2}
	Suppose that
	\begin{equation*}
	\left\{ x \in \Omega\colon \Abs{\bar \Phi(x)} = \alpha \right\}
	\end{equation*} 
	is finite. The elements of this set are denoted by $\bar x^1,\bar x^2, \ldots,\bar x^m$, i.e.,
	\begin{equation*}
	\left\{ x \in \Omega\colon |\bar \Phi(x)| = \alpha \right\} = \left\{\bar x^1,\bar x^2 \ldots, \bar x^m\right\},
	\end{equation*}
	with $m=0$ indicating that these sets are empty.
\end{assumption}
To interpret this assumption recall from \Cref{cor:supportcondition} that
\begin{equation*}
	\supp(\bar q') \subset \left\lbrace x \in \Omega\colon \Abs{\bar \Phi(x)} = \alpha \right\rbrace,
\end{equation*}
hence $\supp(\bar q')$ is also finite. Thus, there exist real numbers $\bar a$ and $\bar c^i$, $1\leq i\leq m$, such that
\begin{equation}\label{eq:barqdarstellung}
	\bar q  = \bar a + \sum_{i=1}^m \bar c^i 1_{(\bar x^i,1)}, \qquad
	\bar q'  = \sum_{i=1}^m \bar c^i \delta_{\bar x^i},
\end{equation}
where some of the coefficients may be zero.
In addition, \eqref{eq:jumpsatrootsofz} yields $\bar z(\bar x^i) = 0$, $1\leq i\leq m$, i.e., the $\bar x^i$ are roots of the continuous adjoint state. Under a mild additional assumption it is possible to prove that the discrete adjoint state $\bar z_h$ admits roots $\bar x_h^i$ close to the $\bar x^i$. Specifically, the distance $\lvert \bar x^i - \bar x_h^i \rvert$ is of order $\mathcal{O}(h)$.

The additional assumption reads as follows.

\begin{assumption}\label{A3}
	Let \Cref{A2} be fulfilled and suppose $\bar z'(\bar x^i) \neq 0$ for $i=1,2,\ldots,m$.
\end{assumption}

We point out that \Cref{A3} is equivalent to the existence of numbers $\kappa>0$ and $R>0$ such that
$|\bar\Phi(x)|\leq \alpha - \kappa |x-\bar x^i|^2$ for all $x\in B_R(\bar x^i)$, $1\leq i\leq m$. That is, \Cref{A3} imposes a \emph{quadratic growth condition} on $\bar\Phi$ near its extreme points $\bar x^i$. Also note that the discrete counterparts $\bar \Phi_h$ and $\hat \Phi_h$ of $\bar \Phi$ are piecewise quadratic functions.

Let us now prove the existence of unique roots of the discrete adjoint state in small neighborhoods of the points $\bar x^i$.
\begin{lemma} \label{lem_discretezerosofz}
	If \Cref{A3} is fulfilled, then there exist $R,\delta,h_0>0$ such that the following holds for all $h\in (0,h_0]$ and all $i=1,2,\ldots,m$. 
	$|\bar z^\prime| \geq \delta$ on $B_R(\bar x^i)$ and $\bar z_h$ has a unique root $\bar x_h^i$ in $B_R(\bar x^i)$. In addition, there hold $B_R(\bar x^i)\cap\partial\Omega=\emptyset$, the $B_R(\bar x^i)$ are pairwise disjoint, and the roots $\bar x_h^i$ satisfy $\abs{\bar x^i-\bar x_h^i}\leq Ch$ for a constant $C>0$ that does not depend on $h$.
\end{lemma}

\begin{proof}
	We first note that $\bar x^i\in\Omega$ is satisfied for $i=1,2,\ldots,m$ since $\bar\Phi(x)=0$ for $x\in\partial\Omega$, whereas
	$\abs{\bar\Phi(\bar x^i)}=\alpha>0$ for $i=1,2,\ldots,m$. Hence, we can assume without loss of generality that $R>0$ is chosen so small that $B_R(\bar x^i)\subset\Omega$ for $i=1,2,\ldots,m$.
	Moreover, we can choose $R>0$ so small that all $B_R(\bar x^i)$ are pairwise disjoint. 
	Thus, it is sufficient to argue for one $i\in \{ 1,2,\dots, m\}$. We write $\bar x:= \bar x^i$ for this $i$.
	 
	Since $\bar z\in\Htwo$, we have $\bar z'\in C(\bar\Omega)$. Thus, \Cref{A3} implies the existence of $R>0$ and $\delta>0$ such that $\bar x$ is the only solution of $\bar z(x)=0$ in $B_R(\bar x)$ and such that $\abs{\bar z'(x)}\geq \delta>0$ for all $x\in B_R(\bar x)$. Since $\bar z'$ is continuous, this inequality implies that $\bar z'$ does not change sign in $B_R(\bar x)$, hence $\bar z$ is strictly monotone in $B_R(\bar x)$.
	In view of \Cref{lem_aprioridiffzzhgrad} we can also achieve that $\bar z_h'$ has for all sufficiently small $h$ the same sign as $\bar z'$ a.e. in $B_R(\bar x)$. 
	Hence, $\bar z_h'$ is either positive or negative almost everywhere in $B_R(\bar x)$.
	
	Evidently, the strict monotonicity of $\bar z$ implies that $\bar z$ assumes both negative and positive values in $B_R(\bar x)$. 
	Fix $x_-,x_+\in B_R(\bar x)$ with $\bar z(x_{-})<0$ and $\bar z(x_{+})>0$.
	Using \Cref{lem_aprioridiffzzh} it follows that for $h_0>0$ sufficiently small $\bar z_h(x_-)<0$ and $\bar z_h(x_+)>0$ for all $h\in(0,h_0]$. Thus, the intermediate value theorem implies for every $h\in(0,h_0]$ the existence of $\bar x_h\in(x_-,x_+)$ with $\bar z_h(\bar x_h)=0$ as claimed.
	
	Suppose that there were an additional root $\hat x_h$ of $\bar z_h$ in $B_R(\bar x)$.
	Then, by the fundamental theorem of calculus for Sobolev functions, 
	we obtain $0=\bar z_h(\bar x_h)-\bar z_h(\hat x_h)=\int_{\hat x_h}^{\bar x_h} \bar z_h'(x) \, dx$. However,
	since $\bar z_h'$ is either positive or negative 
	almost everywhere in $B_R(\bar x)$, this cannot be true. Hence,	$\bar x_h$ is indeed the only root of $\bar z_h$ in $B_R(\bar x)$.
	
	It remains to establish the estimate $\abs{\bar x-\bar x_h}\leq Ch$. 
	Using $0 =\bar z(\bar x) = \bar z_h(\bar x_h)$ and the mean value theorem yields
	\begin{equation*}
		\bar z'(\xi)(\bar x-\bar x_h) = \bar z(\bar x) -\bar z(\bar x_h) = \bar z_h(\bar x_h) - \bar z(\bar x_h)
	\end{equation*}
	for a $\xi\in B_R(\bar x)$.
	Taking absolute values and using $1/\abs{\bar z'(\xi)} \leq 1/\delta$ this implies
	$\abs{\bar x-\bar x_h}\leq \abs{\bar z_h(\bar x_h) - \bar z(\bar x_h)}/\delta\leq Ch/\delta$,
	where we applied \Cref{lem_aprioridiffzzh} again.
\end{proof}

In the next lemma we conclude that in the neighborhoods $B_R(\bar x^i)$ only the $\bar x_h^i$ can satisfy $|\bar\Phi_h(x)|=\alpha$ and that there cannot be any points outside these neighborhoods where $|\bar\Phi_h(x)|=\alpha$ holds.
\begin{lemma} \label{lem_discretenumber}
Suppose that \Cref{A3} is valid and let $R>0$ and $\bar x_h^i$, $1\leq i\leq m$, be as in \Cref{lem_discretezerosofz}. Then there is $h_0>0$ such that for all $h\in (0,h_0]$ and all $x\in\bar\Omega$ we have
\begin{equation*}
\Abs{\bar \Phi_h(x)}=\alpha \quad\enspace\Longrightarrow\quad\enspace
x=\bar x_h^i \, \text{ for some } i\in\{1,2,\ldots,m\}.
\end{equation*}
\end{lemma}

\begin{proof}
Let $h_0>0$ be from \Cref{lem_discretezerosofz} and let $h\in (0,h_0]$ and $x\in\bar\Omega$ be such that $|\bar \Phi_h(x)|=\alpha$. From \Cref{cor:discretesupportcondition} we know that $\bar z_h(x) = 0$. 
We distinguish two cases.

\textbf{Case 1:} $x\in B_R(\bar x^i)$ for some $i\in\{1,2,\ldots,m\}$

In this case the claim follows from \Cref{lem_discretezerosofz}.

\textbf{Case 2:} $x\in\bar\Omega\setminus\bigcup_{i=1}^m B_R(\bar x^i)$

It is sufficient to show that in this case, $\abs{\bar \Phi_h(x)}=\alpha$ cannot be satisfied. 
To this end, we will demonstrate that there is $\epsilon>0$ such that $|\bar\Phi(x)| \leq \alpha-\epsilon$ for all $x\in\bar\Omega\setminus\bigcup_{i=1}^m B_R(\bar x^i)$. 
Granted this claim, we infer from the definitions of $\bar\Phi$ and $\bar\Phi_h$ together with \Cref{lem_aprioridiffzzh} and $|\Omega|=1$ that
\begin{equation*}
\Norm{\bar \Phi - \bar \Phi_h}{\Linf} \leq \Norm{\bar z - \bar z_h}{\Lone} \leq \Norm{\bar z - \bar z_h}{\Linf} \leq Ch.
\end{equation*}
Thus we obtain, for $h$ sufficiently small, that $|\bar \Phi_h(x)|\leq \alpha-\frac{\epsilon}{2}$ for all $x\in\bar\Omega\setminus\bigcup_{i=1}^m B_R(\bar x^i)$ proving that $\abs{\bar \Phi_h(x)}\neq\alpha$ for all these $x$, as desired.

To establish the existence of said $\epsilon$, note that $\abs{\bar\Phi}$ is continuous on the compact set
$\bar\Omega\setminus\bigcup_{i=1}^m B_R(\bar x^i)$. Hence, it attains a maximum on this set, and from \Cref{A2} and $\lVert\bar\Phi\rVert_{\Linf}\leq\alpha$, cf.~\Cref{thm_optcond}, 
it is evident that this maximum is smaller than $\alpha$,
which shows that the desired $\epsilon$ exists, thereby concluding the proof.
\end{proof}

Lemmas \ref{lem_discretezerosofz} and \ref{lem_discretenumber} guarantee the existence of $m$ well-defined pairs $(\bar x^i,\bar x_h^i)$ that are roots of the continuous and discrete adjoint state, respectively. 
By \Cref{cor:supportcondition} and \Cref{cor:discretesupportcondition} we have
\begin{align*}
\supp(\bar q') & \subset \left\lbrace x \in \Omega\colon |\bar \Phi(x)| = \alpha \right\rbrace \subset \big\lbrace x\in \Omega\colon \bar z(x) = 0\big\rbrace,\\
\supp(\bar q_h') & \subset \left\lbrace x \in \Omega\colon |\bar \Phi_h(x)| = \alpha \right\rbrace \subset \big\lbrace x\in \Omega\colon \bar z_h(x) = 0\big\rbrace.
\end{align*}
Therefore, Lemmas \ref{lem_discretezerosofz} and \ref{lem_discretenumber} together with 
\Cref{A2} imply that the number of points of the support of $\bar q'$ and $\bar q_h'$ are both bounded by $m$.  
Using \Cref{lem_discretezerosofz} we observe for the cardinality of the involved sets that
\begin{equation*}
\# \left\lbrace x\in \bigcup_{i=1}^m B_R(\bar x^i)\colon \bar z_h(x) = 0 \right\rbrace  =
\# \left\lbrace x\in \bigcup_{i=1}^m B_R(\bar x^i)\colon \bar z(x) = 0 \right\rbrace = m.
\end{equation*}
Yet, by virtue of \Cref{lem_discretenumber} this implies
\begin{equation*}
\# \left\lbrace x \in \Omega\colon |\bar \Phi_h(x)| = \alpha \right\rbrace \leq \# \left\lbrace x \in \Omega\colon |\bar \Phi(x)| = \alpha \right\rbrace = m,
\end{equation*}
but it can happen, at least for large $h$, that
\[ \#\supp(\bar q')<\#\supp(\bar q_h'). \]
Since we know from
\Cref{cor:discretesupportcondition} and \Cref{lem_discretenumber} that
\begin{equation*}
\supp(\bar q_h') \subset \left\lbrace x \in \Omega\colon \Abs{\bar\Phi_h(x)} = \alpha\right\rbrace
\qquad\text{and}\qquad
\# \left\lbrace x \in \Omega\colon \Abs{\bar\Phi_h(x)} = \alpha\right\rbrace \leq m,
\end{equation*}
we find the following discrete analogue to the continuous representation \eqref{eq:barqdarstellung}:
There exist real numbers $\bar a_h$ and $\bar c_h^i$, $1\leq i\leq m$, such that
\begin{equation}\label{eq:barqhpdarstellung2}
\bar q_h  = \bar a_h + \sum_{i=1}^m \bar c_h^i 1_{(\bar x_h^i,1)}, \qquad
\bar q_h'  = \sum_{i=1}^m \bar c_h^i \delta_{\bar x_h^i}.
\end{equation}
Note that some of the coefficients may be zero.
In addition, we recall that $\bar z_h(\bar x_h^i) = 0$ for $i=1,\ldots,m$ by definition, cf.~\Cref{lem_discretezerosofz}.

Next we estimate the difference between the jump heights of the optimal control $\bar q$ and its counterpart $\bar q_h$.
\begin{lemma} \label{lem_estimateforthejumps}
Suppose that \Cref{A3} is valid. 
Then there exist $C,h_0>0$ such that for all $h\in(0,h_0]$ the optimal controls $\bar q= \bar a + \sum_{i=1}^m \bar c^i 1_{(\bar x^i,1)}$ of \eqref{prob:BV} and $\bar q_h = \bar a_h + \sum_{i=1}^m \bar c^i_h 1_{(\bar x_h^i,1)}$ of \eqref{prob_BVh} satisfy
\begin{equation*}
\sum_{i=1}^m |\bar c^i-\bar c^i_h| \leq C \left( h^2 +  \lVert \bar u - \bar u_h \rVert_{\LL} \right).
\end{equation*}
\end{lemma}

\begin{proof}
Let $R,h_0>0$ be the quantities from \Cref{lem_discretezerosofz}.
Then for all $h\in(0,h_0]$ the balls $B_{\frac{3}{4}R}(\bar x^i)$ are contained in $\Omega$ and are pairwise disjoint for $i=1,\ldots,m$. 
For any $1\leq i\leq m$ we can thus choose a function $g\in C_c^\infty(\Omega)$ such that $g=1$ on $B_{\frac{R}{2}}(\bar x^i)$ and $g = 0$ on $\bar\Omega\setminus B_{\frac34 R}(\bar x^i)$. For $h$ small enough we have $\bar x_h^i\in B_{\frac{R}{2}}(\bar x^i)$ for all $i\in\{1,2,\dots,m\}$ by \Cref{lem_discretezerosofz}. 

Using the structure of the optimal controls, the definition of the distributional derivative,
and the definition of the state equation, we infer for all $h\in(0,h_0]$ that
\begin{equation}\label{eq_maineqinlem4.8}
\begin{split}
\lvert\bar c^i-\bar c^i_h\rvert
& = \lvert \langle \bar q' - \bar q_h', g \rangle_{\M,C(\bar \Omega)}\rvert
  = \lvert - \left( \bar q - \bar q_h, g' \right)_{\LL}\rvert\\
&\leq \Bigl\lvert\left( \bar q - \bar q_h, R_h(g') \right)_{\LL}\Bigr\rvert + \Bigl\lvert\left( \bar q - \bar q_h, g'-R_h(g') \right)_{\LL}\Bigr\rvert.
\end{split}
\end{equation}
For the second term on the right-hand side we observe 
\begin{equation*}
\Bigl\lvert\left( \bar q - \bar q_h, g'-R_h(g') \right)_{\LL}\Bigr\rvert
\leq \left( \Norm{\bar q}{\LL}+\Norm{\bar q_h}{\LL} \right)\Norm{g' - R_h (g')}{\LL}\leq Ch^2
\end{equation*}
due to \Cref{lem_aprioridiffSSh} and the boundedness of $\bar q_h$ independent of $h$ (after decreasing $h_0$ if necessary), cf. \Cref{thm:existenceofsolh}. Using the state equation for the first term we obtain 
\begin{equation*}
\begin{split}
& \Bigl\lvert\left( \bar q - \bar q_h, R_h(g') \right)_{\LL}\Bigr\rvert
= \Bigl\lvert\cala(\bar u-\bar u_h,R_h(g'))\Bigr\rvert \\
& \quad \leq \Bigl\lvert\cala(\bar u,R_h(g')-g')+\cala(\bar u,g')-\cala(\bar u_h,g')\Bigr\rvert \leq \Bigl\lvert\cala(\bar u,R_h(g') -g') \Bigr\rvert + \Bigl\lvert\cala(\bar u - \bar u_h,g')\Bigr\rvert \\
& \quad = \left| (\bar q, R_h(g')-g')_{\LL}\right| + \Bigl\lvert \int\limits_\Omega  (\bar u-\bar u_h)((ag'')'+d_0g')\,dx \Bigr\rvert \\
& \quad \leq Ch^2+\Norm{\bar u-\bar u_h}{\LL}\Norm{(ag')'+d_0g'}{\LL},
\end{split}
\end{equation*}
where the second inequality is obtained by virtue of \Cref{lem_aprioridiffSSh} and integration by parts.
Inserting the two obtained estimates into \eqref{eq_maineqinlem4.8} yields the assertion after summation.
\end{proof}

From the previous lemma we derive an estimate
for the difference between the offsets and the jump positions of $\bar q$ and $\bar q_h$.

\begin{lemma} \label{lem_estimatesforpartsofthecontrols}
Suppose that \Cref{A3} is valid.
Then there exist $C,h_0>0$ such that for all $h\in(0,h_0]$ the optimal controls $\bar q= \bar a + \sum_{i=1}^m \bar c^i 1_{(\bar x^i,1)}$ and $\bar q_h = \bar a_h + \sum_{i=1}^m \bar c^i_h 1_{(\bar x_h^i,1)}$ satisfy
\begin{equation*}
|\bar a - \bar a_h| + \sum_{i=1}^m |\bar x^i-\bar x_h^i| \leq C \left( h^2 + \lVert \bar u - \bar u_h \rVert_{\LL} \right).
\end{equation*}
\end{lemma}

\begin{proof}
\Cref{lem_discretezerosofz} and \Cref{cor:supportcondition} imply
\begin{equation*}
0 = \bar z(\bar x^i) - \bar z(\bar x_h^i) + \bar z(\bar x_h^i) - \bar z_h(\bar x_h^i) =  \bar z'(\xi^i) (\bar x^i-\bar x_h^i) + \bar z(\bar x_h^i) - \bar z_h(\bar x_h^i)
\end{equation*}
for some $\xi^i$ between $\bar x^i$ and $\bar x_h^i$. By \Cref{lem_discretezerosofz} we also have $\abs{\bar z'}\geq\delta>0$ in a neighborhood of $\bar x^i$ containing $\bar x^i_h$ for $i=1,2,\ldots,m$ for $h$ sufficiently small. 
Thus, by Lemmas \ref{lem_aprioridiffSSh}, \ref{lem_ritz} and \ref{lem_aprioriWoneinf} we find
\begin{equation} \label{eq:proof:barxtobarxh0}
\begin{split}
|\bar x^i-\bar x_h^i| 
& \leq C \lVert \bar z -\bar z_h \rVert_{L^\infty(\Omega)} \\
& \leq C \lVert \bar z -R_h\bar z \rVert_{L^\infty(\Omega)} + C \lVert S_h\bar u -\bar z_h \rVert_{L^\infty(\Omega)} \\
& \leq C \left(h^2+\lVert \bar u -\bar u_h \rVert_{\LL}\right).
\end{split}
\end{equation}
It remains to estimate the difference in the offsets. To this end, we denote $\mathcal{S}:=S^* S$ and $\mathcal{S}_h:=S_h^* S_h$ and observe that
\begin{align*}
\bar z - \bar z_h & = S^*S \left( \bar a + \sum_{i=1}^m \bar c^i 1_{(\bar x^i,1)} \right) - S^*_hS_h \left( \bar a_h + \sum_{i=1}^m \bar c^i_h 1_{(\bar x^i_h,1)} \right) - \left(S^* -S^*_h \right)u_d \\
& = (\bar a - \bar a_h) \mathcal{S} 1 + \bar a_h (\mathcal{S} - \mathcal{S}_h) 1 + \sum_{i=1}^m (\bar c^i - \bar c_h^i) \mathcal{S} 1_{(\bar x^i,1)}  \\
& \quad  + \sum_{i=1}^m \bar c_h^i (\mathcal{S}-\mathcal{S}_h) 1_{(\bar x^i,1)}
+ \sum_{i=1}^m \bar c_h^i \mathcal{S}_h (1_{(\bar x^i,1)} - 1_{(\bar x^i_h,1)}) - (S^*-S^*_h) u_d.
\end{align*}
By Theorems \ref{thm_optcond} and \ref{thm_optcondh} the means of $\bar z$ and $\bar z_h$ vanish. Integration hence shows 
\begin{align*}
0 & = (\bar a-\bar a_h) \int_\Omega \mathcal{S} 1~dx + \bar a_h \int_\Omega (\mathcal{S} - \mathcal{S}_h) 1 \,dx +  \sum_{i=1}^m (\bar c^i-\bar c_h^i) \int_\Omega \mathcal{S} 1_{(\bar x^i,1)} \,dx\\
& \quad + \sum_{i=1}^m \bar c_h^i \int_\Omega (\mathcal{S}-\mathcal{S}_h) 1_{(\bar x^i,1)} \,dx
+ \sum_{i=1}^m \bar c_h^i \int_\Omega \mathcal{S}_h (1_{(\bar x^i,1)} - 1_{(\bar x^i_h,1)}) \,dx \\
& \quad - \int_\Omega (S^*-S^*_h) u_d  \,dx.
\end{align*}
As $S$ is an isomorphism, we have
\begin{equation*}
\int_\Omega \mathcal{S} 1 \,dx
= \int_\Omega S^*S 1 \,dx = \lVert S1 \rVert_{\LL}^2 \neq 0
\end{equation*}
and therefore
\begin{equation} \label{eq:proof:convergenceofoffsets0}
\begin{split}
|\bar a-\bar a_h| & \leq \lVert S1 \rVert_{\LL}^{-2} \biggl( |\bar a_h| \lVert (\mathcal{S}-\mathcal{S}_h) 1 \rVert_{L^1(\Omega)}+ \sum_{i=1}^m |\bar c^i-\bar c^i_h| ~ \lVert \mathcal{S} 1_{(\bar x^i,1)} \rVert_{L^1(\Omega)} \biggr.\\
& \quad + \sum_{i=1}^m |\bar c_h^i| \lVert (\mathcal{S}-\mathcal{S}_h) 1_{(\bar x^i,1)} \rVert_{L^1(\Omega)} + \sum_{i=1}^m |\bar c_h^i| \lVert \mathcal{S}_h (1_{(\bar x^i,1)} - 1_{(\bar x^i_h,1)}) \rVert_{L^1(\Omega)} \\
& \qquad \biggl. + \, \lVert (S^*-S^*_h) u_d \rVert_{L^1(\Omega)} \biggr).
\end{split}
\end{equation}
From \Cref{lem_aprioridiffSSh} and $S=S^*$ we deduce
\begin{align*}
\lVert \mathcal{S} - \mathcal{S}_h \rVert_{\Lin( \LL, \LL)} & \leq \lVert S^* \rVert_{\Lin( \LL, \LL)} \lVert S-S_h \rVert_{\Lin( \LL, \LL)} \\
& \qquad + \lVert S^* - S^*_h \rVert_{\Lin( \LL, \LL)} \lVert S_h \rVert_{\Lin( \LL, \LL)}
\leq Ch^2.
\end{align*}
This, \Cref{lem_aprioridiffSSh} and \eqref{eq:proof:convergenceofoffsets0} yield, with $\lvert \bar c_h\rvert_1 := \sum_{i=1}^m |\bar c^i_h|$,
\begin{align*}
|\bar a-\bar a_h| & \leq C h^2 \left( |\bar a_h| + \lvert \bar c_h\rvert_1 + 1 \right) \\
& \qquad + C \sum_{i=1}^m |\bar c^i - \bar c^i_h| + C \sum_{i=1}^m |\bar c_h^i|  \lVert \mathcal{S}_h( 1_{(\bar x^i,1)} - 1_{(\bar x_h^i,1)}) \rVert_{L^1(\Omega)}.
\end{align*}
We have that $\norm{\mathcal{S}_h}{\Lin(\Lone,\Lone)}=\norm{S_h^\ast S_h}{\Lin(\Lone,\Lone)}\leq C$, since $S_h^*=S_h$ and, by standard energy norm estimates,
\begin{equation*}
\|S_hv\|_{L^1(\Omega)}\leq \|S_hv\|_{H_0^1(\Omega)}\leq C\|v\|_{H^{-1}(\Omega)}\leq C\|v\|_{L^1(\Omega)}
\end{equation*}
in one space dimension. We can therefore continue the estimate by
\begin{equation} \label{eq:proof:convergenceofoffsets1}
|\bar a-\bar a_h| \leq C h^2 \left( |\bar a_h| + |\bar c_h|_1 + 1 \right) + C \sum_{i=1}^m |\bar c^i - \bar c^i_h| + C \sum_{i=1}^m |\bar c_h^i|  \lVert 1_{(\bar x^i,1)} - 1_{(\bar x_h^i,1)} \rVert_{L^1(\Omega)}.
\end{equation}
From the definition of $\bar q_h$ we obtain
\begin{equation*}
\frac{1}{2}\lVert \bar a_h S1 + \sum_{i=1}^m \bar c_h^i S1_{(\bar x^i_h,1)} - u_d \rVert_{\LL}^2 + \alpha \lVert \sum_{i=1}^m \bar c_h^i \delta_{\bar x^i_h} \rVert_{\M} =  j_h(\bar q_h) \leq j_h(0) = j(0).
\end{equation*}
This implies $\lvert \bar c_h\rvert_1 = \lVert \sum_{i=1}^m \bar c_h^i \delta_{\bar x^i_h} \rVert_{\M} \leq C$ and because of $S1\neq 0$ it also yields $|\bar a_h|\leq C$ with constants independent of $h$. By \Cref{lem_estimateforthejumps} we have $\sum_{i=1}^m |\bar c^i-\bar c_h^i|\leq C( h^2 + \lVert \bar u - \bar u_h \rVert_{\LL} )$. Obviously, it also holds that $\lVert 1_{(\bar x^i,1)} - 1_{(\bar x_h^i,1)} \rVert_{L^1(\Omega)} = |\bar x^i-\bar x_h^i|$. Thus, \eqref{eq:proof:barxtobarxh0} and \eqref{eq:proof:convergenceofoffsets1} show
\begin{equation*}
|\bar a-\bar a_h| \leq C h^2 + C \lVert \bar u -\bar u_h \rVert_{\LL}.\qedhere
\end{equation*}
\end{proof}

The previous two results have the following consequence.

\begin{corollary}\label{cor_estimatesforcontrols}
Suppose that \Cref{A3} is valid.
Then there exist $C,h_0>0$ such that for all $h\in(0,h_0]$ we have
\begin{equation*}
\lVert \bar q - \bar q_h \rVert_{L^1(\Omega)} \leq C\left(  h^2 + \lVert \bar u - \bar u_h \rVert_{\LL} \right).
\end{equation*}
\end{corollary}

\begin{proof}
As
\begin{equation*}
\lVert\bar q_h - \bar q\rVert_{L^1(\Omega)}
\leq \lvert\bar a_h -\bar a\rvert |\Omega| + 
\sum_{i=1}^m \lvert\bar c^i_h - \bar c_i\rvert \lVert 1_{(\bar x_h^i,1)} \rVert_{L^1(\Omega)}
+\sum_{i=1}^m \lvert\bar c^i\rvert \lVert 1_{(\bar x_h^i,1)}-1_{(\bar x^i,1)} \rVert_{L^1(\Omega)},
\end{equation*}	
the result follows from \Cref{lem_estimateforthejumps} together with \Cref{lem_estimatesforpartsofthecontrols}.
\end{proof}

%


In view of \Cref{cor_estimatesforcontrols}
it remains to estimate $\lVert \bar u-\bar u_h\rVert_{\LL}$. 
We are now able to establish convergence order $h^2$ for the optimal state.

\begin{theorem}\label{thm_errorestimatestatesorderh2}
Suppose that \Cref{A3} is valid.
Then there exist $C,h_0>0$ such that for all $h\in(0,h_0]$ we have
\begin{equation*}
\lVert \bar u_h -\bar u\rVert_{\LL} \leq C h^2.
\end{equation*}
\end{theorem}

\begin{proof}
Combining \Cref{lem_1steptoh2convforthestates} with Hölder's inequality and \Cref{cor_estimatesforcontrols} leads to
\begin{equation*}
\lVert \bar u_h -\bar u\rVert_{\LL}^2 \leq C \lVert \bar z -R_h \bar z \rVert_{L^\infty(\Omega)} \left(h^2 + \lVert \bar u - \bar u_h \rVert_{\LL} \right) + Ch^4. 
\end{equation*}
By Young's inequality this yields
\begin{equation*}
\frac{1}{2}\lVert \bar u_h -\bar u\rVert_{\LL}^2 \leq C \lVert \bar z - R_h \bar z \rVert_{L^\infty(\Omega)} h^2 + 
C \lVert \bar z - R_h \bar z \rVert_{L^\infty(\Omega)}^2 + Ch^4.
\end{equation*} 
Since $\bar z\in W^{2,\infty}(\Omega)$, the error estimate of the Ritz projection from Lemma \ref{lem_ritz} thus implies the assertion.
\end{proof}

Finally, we obtain convergence of order $h^2$ also for the optimal control and the optimal adjoint state, but with respect to the $L^1(\Omega)$-norm and the $L^\infty(\Omega)$-norm, respectively.

\begin{corollary}\label{cor_errorestimatestatesorderh2}
Suppose that \Cref{A3} is valid. Then there exist $C,h_0>0$ such that for all $h\in(0,h_0]$ we have the following estimates of the structural differences of $\bar q$ and $\bar q_h$
\begin{equation*}
\sum_{i=1}^m |\bar x^i - \bar x^i_h| \leq Ch^2, \qquad \sum_{i=1}^m |\bar c^i - \bar c^i_h| \leq Ch^2 \qquad \text{ and } \qquad |\bar a - \bar a_h| \leq Ch^2.
\end{equation*}
We also have the error estimates
	\begin{equation*}
		\Norm{\bar q-\bar q_h}{\Lone}\leq Ch^2 \qquad\text{ and }\qquad
		\Norm{\bar z-\bar z_h}{\Linf}\leq Ch^2.
	\end{equation*}		
\end{corollary}

\begin{proof}
For the first four claims combine \Cref{lem_estimateforthejumps}, \Cref{lem_estimatesforpartsofthecontrols} and \Cref{cor_estimatesforcontrols} with \Cref{thm_errorestimatestatesorderh2}.
For the last claim note that \Cref{lem_aprioridiffSSh} and \Cref{lem_aprioriS} imply, due to the embedding $H_0^1(\Omega)\hookrightarrow L^\infty(\Omega)$,
	\begin{equation*}
	\begin{split}
		\Norm{\bar z - \bar z_h}{\Linf} 
		& \leq \Norm{\bar z - R_h\bar z}{\Linf} + \Norm{R_h\bar z - \bar z_h}{\Linf}\\
		& = \Norm{\bar z - R_h\bar z}{\Linf} + \Norm{S_h\bar u-S_h\bar u_h}{\Linf}\\
		& \leq \Norm{\bar z - R_h\bar z}{\Linf}  \\
		& \qquad + C\Norm{(S_h-S)(\bar u-\bar u_h)}{H^1_0(\Omega)} + C\Norm{S(\bar u-\bar u_h)}{H^1_0(\Omega)}\\
		& \leq \Norm{\bar z - R_h\bar z}{\Linf} + Ch^2,
	\end{split}
	\end{equation*}
	where we have also used \Cref{thm_errorestimatestatesorderh2} to deduce the last inequality. The claim follows by taking into account the Ritz projection error from \Cref{lem_ritz}.
\end{proof}

\subsection{Error estimates for piecewise constant control discretization}\label{subsec_error_fd}

In this section we prove convergence rates for \eqref{prob_BVch}.
Let us stress that we can only expect $\lVert \hat q_h - \bar q \rVert_{L^1(\Omega)} = \mathcal{O}(h)$ because for $\bar q = 1_{(\bar x,1)}$, $\bar x\in\Omega$, we have $\lVert 1_{(x_j,1)} - 1_{(\bar x,1)} \rVert_{L^1(\Omega)} = |x_j-\bar x| = \mathcal{O}(h)$ for any node $x_j$. 
We will establish precisely this order of convergence and emphasize that the numerical experiments in Sections \ref{sec:academicaexample} and \ref{sec:naturalexample} indicate that this order is indeed optimal.

As in the variationally discrete case we begin by establishing an error estimate for the state and the adjoint state that holds without any structural assumption on the optimal controls. In fact, we are not able to improve this further. Still, in a second step we can derive an error estimate for the control relying on the same structural assumptions as in the variationally discretized setting.
 
\subsubsection{Basic error estimates for state and adjoint equation}

\begin{lemma} \label{prop:suboptimalstateerrorch}
	Let $h_0>0$ be as in \Cref{thm:existence_probch}. For any $h\in (0,h_0]$ the optimal state $\hat u_h$ associated with the optimal control $\hat q_h$ to \eqref{prob_BVch} satisfies
	\begin{equation*}
	\lVert \hat u_h - \bar u \rVert_{L^2(\Omega)} \leq Ch
	\end{equation*}
	with a constant $C$ independent of $h$.
\end{lemma}

\begin{proof}
	By \Cref{thm:existence_probch} we have that for any $h\in (0,h_0]$ there exists a unique optimal control $\hat q_h$ to \eqref{prob_BVch} with associated state $\hat u_h$ and adjoint state $\hat z_h$. We test the variational inequality from \Cref{thm_optcondch} with $q_h=\Pi_h\bar q\in Q_h$ and the variational inequality from \Cref{thm_optcond} with $q=\hat q_h$ and obtain
	\begin{align*}
	- \left( \hat z_h, \Pi_h \bar q - \hat q_h \right)_{L^2(\Omega)} + \alpha \lVert \hat q_h' \rVert_{\M} & \leq \alpha \lVert (\Pi_h\bar q)' \rVert_{\M},\\
	- \left( \bar z, \hat q_h - \bar q \right)_{L^2(\Omega)} + \alpha \lVert \bar q' \rVert_{\M} & \leq \alpha \lVert \hat q_h' \rVert_{\M}.
	\end{align*}
	Adding those two lines and using \Cref{lem:gradnormstabilityofPih} we find
	\begin{align*}
	- \left( \hat z_h, \Pi_h \bar q - \hat q_h \right)_{L^2(\Omega)} - \left( \bar z, \hat q_h - \bar q \right)_{L^2(\Omega)} +  \alpha \lVert \hat q_h' \rVert_{\M} + \alpha \lVert \bar q' \rVert_{\M} \\
	\leq \alpha \lVert (\Pi_h\bar q)' \rVert_{\M} + \alpha \lVert \hat q_h' \rVert_{\M} \leq \alpha \lVert \bar q' \rVert_{\M} + \alpha \lVert \hat q_h' \rVert_{\M}.
	\end{align*}
Rearranging terms and using \eqref{eq:orthogonality} leads to
\begin{equation*}
	\left(\bar z- \hat z_h,\bar q - \hat q_h \right)_{L^2(\Omega)} \leq  \left( \hat z_h, \Pi_h \bar q - \bar q \right)_{L^2(\Omega)} =  \left( \hat z_h - \Pi_h \hat z_h, \Pi_h \bar q - \bar q \right)_{L^2(\Omega)}.
\end{equation*}
By \Cref{lem:Piherrorestimate} we obtain
\begin{equation*}
	\left(\bar z- \hat z_h,\bar q - \hat q_h \right)_{L^2(\Omega)} \leq \|\Pi_h\bar q-\bar q\|_{L^1(\Omega)}\|\Pi_h\hat z_h-\hat z_h\|_{L^\infty(\Omega)}\leq h^2 \lVert \hat q_h^\prime\rVert_{\M} \lVert \hat z_h^\prime\rVert_{L^\infty(\Omega)}.
\end{equation*}
Using \Cref{thm:existenceofsolh}, \Cref{lem_aprioriWoneinf} and the boundedness of $\lVert \hat u_h - u_d \rVert_{\LL}$, which is due to \Cref{thm:existence_probch}, we find $\lVert \hat q_h^\prime \rVert_{\M}, \lVert \hat z_h^\prime \rVert_{L^\infty(\Omega)} \leq C$ and thus 
\begin{equation} \label{eq:adjointestimate}
	\left(\bar z- \hat z_h,\bar q - \hat q_h \right)_{L^2(\Omega)} \leq C h^2.
\end{equation}
We introduce the auxiliary state $\tilde u_h:=S_h\bar q$ and observe with the boundedness results from \Cref{thm:existenceofsol} and \Cref{thm:existence_probch} together with \Cref{lem_aprioridiffSSh} that
\begin{equation*}
	\begin{split}
	\lVert \tilde u_h - \hat u_h \rVert_{L^2(\Omega)}^2= &  \left( S_h (\bar q - \hat q_h ), S_h(\bar q - \hat q_h) \right)_{L^2(\Omega)}\\=&\left( S_h^*( S_h \bar q - S_h \hat q_h ), \bar q - \hat q_h \right)_{L^2(\Omega)}\\
	=&\left( S^*( S\bar q -u_d), \bar q - \hat q_h \right)_{L^2(\Omega)}-\left( S_h^*( S_h\hat q_h -u_d), \bar q - \hat q_h \right)_{L^2(\Omega)}\\&-\left( S^*( S-S_h)\bar q, \bar q - \hat q_h \right)_{L^2(\Omega)}-\left( (S_h^*-S^*)u_d, \bar q - \hat q_h \right)_{L^2(\Omega)}\\
	&-\left( (S^*-S_h^*)S_h \bar q , \bar q - \hat q_h \right)_{L^2(\Omega)}\\\leq& \left( \bar z - \hat z_h, \bar q - \hat q_h \right)_{L^2(\Omega)}+Ch^2,
	\end{split}
\end{equation*}
pointing out that due to $S^*=S$ and $S_h^*=S_h$ the same finite element discretization error estimates as for the state equation apply to the adjoint states.
Combining this with \eqref{eq:adjointestimate} leads to $\|\tilde u_h-\hat u_h\|_{L^2(\Omega)}\le Ch$.
Therefore, the assertion follows from
\begin{equation*}
\|\bar u-\hat u_h\|_{L^2(\Omega)}\le\|\bar u-\tilde u_h\|_{L^2(\Omega)}+\|\tilde u_h-\hat u_h\|_{L^2(\Omega)}\le Ch,
\end{equation*}
where the first summand is of order $h^2$ by \Cref{lem_aprioridiffSSh}.
\end{proof}

The preceding lemma has the following consequence.

\begin{corollary}\label{prop:adjointconv1}
	Let $h_0>0$ be from \Cref{thm:existence_probch} and $h\in (0,h_0]$. Let $(\hat u_h,\hat q_h, \hat z_h )$ be the optimal triple of \eqref{prob_BVch} and $(\bar u,\bar q,\bar z)$ the optimal triple of \eqref{prob:BV}. Then there holds
	\begin{equation*}
	\lVert \bar z - \hat z_h \rVert_{W^{1,\infty}(\Omega)}\le Ch
	\end{equation*}
	with a constant $C>0$ independent of $h$.
\end{corollary}

\begin{proof}
The proof is essentially the same as for Lemma~\ref{lem_aprioridiffzzh}, with Lemma~\ref{prop:suboptimalstateerrorch} replacing Lemma~\ref{lem_aprioridiffuuh}.
\end{proof}

\subsubsection{Improved error estimates under structural assumptions}

Similarly as in the variationally discrete setting we will now 
use the structural Assumptions~\ref{A2} and \ref{A3} to derive an $\Lone$-error estimate for the control. 
We recall that Assumption~\ref{A2} ensures that 
$\bar\Phi$ has only finitely many minima and maxima, which in turn implies 
that the optimal control exhibits only finitely many jumps. 
The main idea underlying the proof of the error estimate is to examine the distance between jump points and jump heights of the continuous and the discrete optimal control. Note that the discrete optimal control $\hat q_h$ is piecewise constant and can only admit jumps at the gridpoints $x_j$ with $|\hat\Phi_h(x_j)|=\alpha$. 
These jumps can only occur close to points where $|\bar\Phi|=\alpha$, i.e., in the vicinity of the $\bar x^i$, $i=1,2,\ldots,m$, as the following result shows. 

\begin{lemma}\label{lem_absvalofhatPhihssalpha}
	Suppose that \Cref{A2} is valid and let $R>0$ be as in \Cref{lem_discretezerosofz}. Then there is $h_0>0$ such that for all $h\in (0,h_0]$ and all $x\in\bar\Omega\setminus\bigcup_{i=1}^m{B_{\frac{R}{2}}(\bar x_i)}$ we have
	\begin{equation*}
	\Abs{\hat \Phi_h(x)}<\alpha.
	\end{equation*}
\end{lemma}

\begin{proof}
	The proof follows along the lines of Case 2 in Lemma \ref{lem_discretenumber}.
\end{proof}

Next we investigate the behavior of $\hat\Phi_h$ 
inside the balls $B_R(\bar x^i)$. Note that if $|\hat\Phi_h|<\alpha$ in $B_R(\bar x_i)$, then $\hat q_h$ will not admit a jump in $B_R(\bar x^i)$, hence $\hat c_h^j=0$ in \eqref{eq:barqhpdarstellung} for all $j$ with $x_j\in B_R(\bar x^i)$.
We therefore consider points where $|\hat\Phi_h|\geq \alpha$ and remark that points with $|\hat\Phi_h|>\alpha$ can actually exist because $\hat\Phi_h$ is piecewise quadratic.
\begin{lemma} \label{lem:fullydiscretezerosofzch}  
	Let \Cref{A3} hold and let $R >0$ be as in Lemma \ref{lem_discretezerosofz}. There exists an $h_0>0$ such that the following holds for all $h\in (0,h_0]$. If $|\hat\Phi_h(\hat x)|\ge \alpha$ for some $\hat x\in B_R(\bar x^i)$ and some $i\in\{1,2,\dots,m\}$, then $\hat z_h$ has a unique root $\hat x_h^i$ in $B_{R}(\bar x^i)$ and there holds $\hat x_h^i\in B_{\frac{R}{2}}(\bar x^i)$.
	Moreover, the point $\hat x^i_h$ is the unique local maximizer of $|\hat\Phi_h|$ in $B_R(\bar x^i)$ and satisfies $|\hat\Phi_h(\hat x_h^i)|\geq \alpha$ and $|\bar x^i-\hat x^i_h|\leq Ch$ with a constant $C$ not depending on $h$.
\end{lemma}

\begin{proof}
Without loss of generality let us assume that $h_0\leq R/2$. 
We argue for the case $\hat\Phi_h(\hat x) \geq \alpha$ for some $\hat x\in B_R(\bar x^i)$ and an $i\in\{1,2,\ldots,m\}$. The case $\hat\Phi_h(\hat x) \leq -\alpha$ can be handled analogously. 
Due to $\hat\Phi_h(\hat x) \geq \alpha$ we infer from \Cref{lem_absvalofhatPhihssalpha} that
$\hat x\in B_{\frac{R}{2}}(\bar x^i)$. 
Since $h_0\leq R/2$, we find gridpoints $\hat x^i_{l,h}$ and $\hat x^i_{r,h}$ that satisfy $\hat x^i_{l,h},\hat x^i_{r,h}\in B_R(\bar x^i)$ and $\hat x^i_{l,h} < \hat x < \hat x^i_{r,h}$.
Since $\hat\Phi_h$ 
satisfies $|\hat\Phi_h(x_j)|\leq \alpha$ for all $0\leq j\leq l$, cf. \Cref{thm_optcondch}, 
we have 
\begin{equation*}
\hat\Phi_h(\hat x^i_{l,h}), \hat\Phi_h(\hat x^i_{r,h}) \leq \alpha\leq \hat\Phi_h(\hat x).
\end{equation*}
Hence, the continuous function $\hat\Phi_h$ attains a local maximum at some $\hat x_h^i\in (\hat x^i_{l,h},\hat x^i_{r,h})$.
Clearly, there hold $0 = \hat\Phi_h^\prime(\hat x_h^i) = \hat z_h(\hat x_h^i)$ and 
$\hat\Phi_h(\hat x_h^i) \geq \hat\Phi_h(\hat x) \geq \alpha$, with the latter implying
$\hat x_h^i\in B_{\frac{R}{2}}(\bar x^i)$ by \Cref{lem_absvalofhatPhihssalpha}. 
The uniqueness of the root $\hat x_h^i$ in $B_R(\bar x^i)$ can be established as in the proof of \Cref{lem_discretezerosofz}. The estimate for $|\bar x^i - \hat x_h^i|$ also follows as in the proof of \Cref{lem_discretezerosofz}.
\end{proof}

In the gridpoints we have $|\hat\Phi_h|\le\alpha$. Next we show that $|\hat\Phi_h(x_j)| = \alpha$ for a gridpoint $x_j$ can only hold if $x_j=\hat x_h^i$ for some $i\in\{1,\ldots,m\}$ or if $|\hat\Phi(\hat x_h^i)| > \alpha$ and $x_j$ is close to $\hat x_h^i$.

\begin{corollary} \label{cor:fullydiscjumps}
Let \Cref{A3} hold and let $R$ be as in Lemma \ref{lem:fullydiscretezerosofzch}. 
There exists $h_0>0$ such that the following holds for all $h\in (0,h_0]$. If $|\hat\Phi_h(\hat x)| \geq \alpha$ for some $\hat x\in B_R(\bar x^i)$ and some $i\in \{1,2,\dots,m\}$, then the point $\hat x_h^i \in B_{\frac{R}{2}}(\bar x^i)$ from \Cref{lem:fullydiscretezerosofzch} satisfies exactly one of the following two statements:
\begin{enumerate}
\item $|\hat\Phi_h(\hat x_h^i)| = \alpha$ and $|\hat\Phi_h(y)| < \alpha$ for all $y\in B_R(\bar x^i) \setminus \{\hat x_h^i\}$.
\item $|\hat\Phi_h(\hat x_h^i)| > \alpha$ and there exist exactly two points 
$y_l^i,y_r^i\in B_R(\bar x^i)$ such that $|\hat\Phi_h(y_l^i)|=|\hat\Phi_h(y_r^i)|=\alpha$
and $\hat x_h^i \in (y_l^i, y_r^i)$.
In addition, $y_l^i, y_r^i\in [\hat x_{l,h}^i,\hat x_{r,h}^i]$, where
$\hat x_{l,h}^i,\hat x_{r,h}^i \in B_R(\bar x^i)$ are the gridpoints closest to $\hat x_h^i$ that satisfy $\hat x_{l,h}^i<\hat x_h^i<\hat x_{r,h}^i$. 
Furthermore, there holds $|\hat\Phi_h(y)|<\alpha$ for all $y\in B_R(\bar x_i)\setminus[\hat x_{l,h}^i,\hat x_{r,h}^i]$.
\end{enumerate}
Moreover, we have
\begin{equation} \label{eq:cor:jumpsofqh}
\supp(\hat q_h') \subset \left\lbrace x \in \mathcal{N}_h \colon \Abs{\hat\Phi_h(x)} = \alpha\right\rbrace
\quad\text{and}\quad
\# \left\lbrace x \in \mathcal{N}_h \colon \Abs{\hat\Phi_h(x)} = \alpha\right\rbrace \leq 2m.
\end{equation}
\end{corollary}

\begin{proof}
The first part of \eqref{eq:cor:jumpsofqh} is just a restatement of \Cref{cor_Phihattouchesboundhat} and the second part of \eqref{eq:cor:jumpsofqh} follows from the main statement in combination with 
\Cref{lem_absvalofhatPhihssalpha}.

We assume $\hat\Phi_h(\hat x) \geq \alpha$; the case $\hat\Phi_h(\hat x) \leq - \alpha$ is treated analogously.

We first consider statement 1. Assume $\hat\Phi(\hat x_h^i) = \alpha$. Since $\hat x_h^i$ is the unique maximizer of $\hat\Phi$ on $B_R(\bar x^i)$ by \Cref{lem:fullydiscretezerosofzch} the statement readily follows.

To establish 2, let us assume that $\hat\Phi_h(\hat x_h^i) > \alpha$. Let $x_l,x_r\in B_R(\bar x^i)$ be the two gridpoints closest to $\hat x_h^i$ that satisfy $\hat x_{l,h}^i < \hat x_h^i < \hat x_{r,h}^i$. The existence of such $\hat x_{l,h}^i,\hat x_{r,h}^i$ is ensured if $h_0 < R/2$. 
Since $\hat\Phi_h(\hat x_{l,h}^i), \hat\Phi_h(\hat x_{r,h}^i) \leq \alpha$, the intermediate value theorem implies that there exist $y_l^i\in [\hat x_{l,h}^i, \hat x_h^i)$ and $y_r^i \in (\hat x_h^i, \hat x_{r,h}^i]$ such that $\hat\Phi_h(y_l^i) = \hat\Phi_h(y_r^i) = \alpha$. In particular, we have $\hat x_h^i \in (y_l^i, y_r^i)$.

To demonstrate uniqueness of $y_l^i,y_r^i$ in $B_R(\bar x^i)$, assume there were an $x\in B_R(\bar x^i)\setminus\{y_l^i,y_r^i\}$ with $\hat\Phi_h(x)=\alpha$.
If $x<\hat x_h^i$, then Rolle's theorem yields a $\xi$ between 
$x$ and $y_l^i$ with $\hat z_h(\xi)=0$, which contradicts the uniqueness of the root $\hat x_h^i$ of $\hat z_h$ in $B_R(\bar x_i)$ proven in Lemma \ref{lem:fullydiscretezerosofzch}. If $x>\hat x_h^i$, then we readily obtain a similar contradiction.
Since $x=\hat x_h^i$ is excluded due to $\hat\Phi_h(\hat x_h^i)>\alpha$, we conclude that $\hat\Phi_h(x)=\alpha$ for $x\in B_R(\bar x^i)$ if and only if $x\in\{y_l^i,y_r^i\}$. 

To prove that $|\hat\Phi_h(x)|<\alpha$ for all $x\in B_R(\bar x_i)\setminus[\hat x_{l,h}^i,\hat x_{r,h}^i]$, we assume without loss of generality that $h_0\leq R/4$ so that 
we are able to find gridpoints $\tilde x_{l,h}^i,\tilde x_{r,h}^i\in B_R(\bar x^i)$ with $\tilde x_{l,h}^i<\hat x_{l,h}^i\leq y_l^i$ and $y_r^i\leq \hat x_{r,h}^i<\tilde x_{r,h}^i$. Because we have established that
$\hat\Phi_h(x)=\alpha$ for $x\in B_R(\bar x^i)$ if and only if $x\in\{y_l^i,y_r^i\}$, 
there holds $\hat\Phi_h(\tilde x_{l,h}^i),\hat\Phi_h
(\tilde x_{r,h}^i)<\alpha$. Thus, a continuity argument supplies $\hat\Phi_h(x)<\alpha$ for all $x\in B_R(\bar x^i)\setminus [\hat y_{l}^i,\hat y_{r}^i]$. The claim follows since 
$[\hat y_{l}^i,\hat y_{r}^i]\subset [\hat x_{l,h}^i,\hat x_{r,h}^i]$.
\end{proof}

Summarizing we now know that $\hat q_h$ cannot jump outside of any $B_R(\bar x ^i)$, $1\leq i\leq m$, and that inside every $B_R(\bar x^i)$ jumps can only occur at $\hat x_h^i$ (Case~1) or at any of the two points $y_l^i$ and $y_r^i$ (Case~2), $1\leq i\leq m$. In addition, such a jump can only occur if the respective point is a gridpoint. 
In contrast, in the variational discrete setting the jumps of $\bar q_h$ are not restricted to gridpoints. For clarification we point out that there might well be situations, for large $h$, where the continuous optimal control $\bar q$ jumps at $\bar x^i$, but the discrete optimal control $\hat q_h$ does not admit a jump in $B_R(\bar x^i)$. Vice versa, for large $h$ it may happen that $\hat q_h$ exhibits one or two jumps in $B_R(\bar x^i)$, but $\bar q$ does not jump in $B_R(\bar x^i)$.

To obtain a convergence result, we need to estimate the difference in the jump points and the corresponding coefficients. In the remainder of this section we use the following notation. We write $\hat x_{l,h}^i, \hat x_{r,h}^i$ according to \Cref{cor:fullydiscjumps} if the second case of \Cref{cor:fullydiscjumps} applies. 
If the first case of \Cref{cor:fullydiscjumps} applies, then $\hat x_{l,h}^i, \hat x_{r,h}^i$ denote the left and right neighbor of $\hat x_h^i$, provided that $\hat x_h^i$ itself is not a  gridpoint. If it is a gridpoint, then we denote by $x_{l,h}^i$ its left neighbor and set $\hat x_{r,h}^i:=\hat x_h^i$.
If neither case applies, then we have $|\hat\Phi_h| < \alpha$ in $B_R(\bar x^i)$. In this case, $\hat x_{l,h}^i, \hat x_{r,h}^i$ are taken to be the gridpoints adjacent to each other and satisfying $\bar x^i\in [\hat x_{l,h}^i, \hat x_{r,h}^i)$.
We observe that $|\hat x_{l,h}^i-\hat x_h^i|,|\hat x_{r,h}^i-\hat x_h^i|\leq h$ whenever $\hat x_h^i$ exists, and $|\hat x_{l,h}^i-\bar x^i|,|\hat x_{r,h}^i-\bar x^i|\leq h$ otherwise. In view of \Cref{lem:fullydiscretezerosofzch} this immediately
implies the following result, that does not require a proof.
\begin{lemma} \label{prop:jumpsplitting}
Let \Cref{A3} hold. There exist $C,h_0>0$ such that for all $h\in (0,h_0]$ and $i=1,2,\dots,m$ we have $|\bar x^i - \hat x_{l,h}^i|, |\bar x^i - \hat x_{r,h}^i| \leq Ch$.
\end{lemma}

By virtue of the inclusion in \eqref{eq:cor:jumpsofqh} the preceding discussion furthermore shows that $\hat q_h$ can be represented as follows. There exist real numbers $\hat a_h,\hat c_{l,h}^i,\hat c_{r,h}^i$, $1\leq i\leq m$, such that
\begin{equation}\label{eq:barqhpdarstellung1}
\hat q_h  = \hat a_h + \sum_{i=1}^m (\hat c_{l,h}^i 1_{(\hat x_{l,h}^i,1)}+\hat c_{r,h}^{i}1_{(\hat x_{r,h}^i,1)}),\qquad
\hat q_h'  = \sum_{i=1}^m(\hat c_{l,h}^i \delta_{\hat x_{l,h}^i}+\hat c_{r,h}^{i}\delta_{\hat x_{r,h}^i}),
\end{equation}
where some of the coefficients may be zero. 




We estimate the difference between the jump heights of the optimal control $\bar q$ and its discrete counterpart $\hat q_h$.

\begin{lemma} \label{lem_estimateforthejumps2}
Suppose that \Cref{A3} is valid.
Then there exist $C,h_0>0$ such that for all $h\in(0,h_0]$ the optimal controls $\hat q_h= \hat a_h + \sum_{i=1}^m (\hat c^i_{l,h} 1_{(\hat x_{l,h}^i,1)}+\hat c^i_{r,h} 1_{(\hat x_{r,h}^i,1)} )$ and $\bar q = \bar a + \sum_{i=1}^m \bar c^i 1_{(\bar x^i,1)}$ satisfy
\begin{equation*}
\sum_{i=1}^m |\bar c^i-(\hat c^i_{l,h}+\hat c^{i}_{r,h})| \leq C h.
\end{equation*}
\end{lemma}

\begin{proof}
The proof of Lemma \ref{lem_estimateforthejumps} remains valid for $\hat q_h,\hat u_h,\hat z_h$ and yields
\begin{equation*}
\sum_{i=1}^m |\bar c^i-(\hat c^i_{l,h}+\hat c^{i}_{r,h})| \leq C \left( h^2 +  \lVert \bar u - \hat u_h \rVert_{\LL} \right).
\end{equation*}
Applying \Cref{prop:suboptimalstateerrorch} establishes the desired estimate.
\end{proof}

The difference between the offsets and the jump positions of $\bar q$ and $\hat q_h$ can be estimated as follows.

\begin{lemma} \label{lem_estimatesforpartsofthecontrols_cd}
	Suppose that \Cref{A3} is valid.
	Then there exist $C,h_0>0$ such that for all $h\in(0,h_0]$ the optimal controls $\hat q_h= \hat a_h + \sum_{i=1}^m (\hat c^i_{l,h} 1_{(\hat x_{l,h}^i,1)}+\hat c^i_{r,h} 1_{(\hat x_{r,h}^i,1)} )$ and $\bar q = \bar a + \sum_{i=1}^m \bar c^i 1_{(\bar x^i,1)}$ satisfy
	\begin{equation*}
	|\bar a-\hat a_h| + \sum_{i=1}^m \left( |\bar x^i-\hat x_{l,h}^i|+|\bar x^{i}-\hat x_{r,h}^i| \right) \leq C h.
	\end{equation*}
\end{lemma}

\begin{proof}
	In view of \Cref{prop:jumpsplitting} it only remains to 
	estimate the difference $|\bar a-\hat a_h|$. This can be accomplished almost verbatim as in Lemma \ref{lem_estimatesforpartsofthecontrols}.
\end{proof}

We obtain the following error estimate for the control in $\Lone$.
\begin{corollary}\label{cor_errorestimatestatesorderh_impr}
	Suppose that \Cref{A3} is valid.
	Then there exist $C,h_0>0$ such that for all $h\in(0,h_0]$ we have
	\begin{equation*}
	\Norm{\bar q-\hat q_h}{\Lone}\leq Ch.
	\end{equation*}		
\end{corollary}

\begin{proof}
	The desired estimate follows by combining \Cref{lem_estimateforthejumps2} and \Cref{lem_estimatesforpartsofthecontrols_cd}.
\end{proof}


\section{Numerical Experiments}\label{sec_numerics}

\newlength\fheight
\newlength\fwidth
In this section we introduce an algorithm to solve the optimization problems \eqref{prob_BVh} and \eqref{prob_BVch} 
based on the PDAP method described for example in \cite{PQTTW2018,Daniel2018}. Moreover, we discuss the error estimates for both discretization schemes on two numerical examples.

\subsection{Optimization algorithm for variational discretization}

We recall from \eqref{eq:barqhpdarstellung2} that there is a number $m\in\N$ such that the optimal control $\bar q_h$ for \eqref{prob_BVh} and its derivative can be expressed as $\bar q_h = \bar a_h+\sum_{i=1}^{m} \bar c^i_h 1_{(\bar x_h^i,1)}$, respectively, $\bar q_h' = \sum_{i=1}^{m} \bar c^i_h \delta_{\bar x_h^i}$ for suitable coefficients 
$\bar a_h$, $\bar c_h:=(\bar c_h^1,\ldots,\bar c_h^m)^T\in\R^{m}$ and points $\bar x_h^1,\ldots,\bar x_h^m\in\Omega=(0,1)$ that satisfy $\bar z_h(\bar x_h^i)=0$ for $0\leq i\leq m$. 
Let us assume for a moment that we know $\{\bar x_h^i\}_{i=1}^m$. We can then determine the coefficients $\bar a_h$ and $\bar c_h$ by solving the finite-dimensional, convex optimization problem
\begin{equation}\label{prob_semisemi2}
\min_{a_h\in\R,c_h\in\R^{m}} \frac12\Norm{u_h-u_d}{\LL}^2 + \alpha\sum_{i=1}^{m} |c_h^i|
\qquad\text{s.t.}\qquad
u_h = S_h\left( a_h+\sum_{i=1}^{m} c_h^i 1_{(\bar x_h^i,1)}\right).
\end{equation}
Since we do not know $\{\bar x_h^i\}_{i=1}^m$ beforehand, the algorithmic idea is to work with approximations of this set. 
We start with an approximation $\{t_{(0)}^i\}_{i=0}^{m_{(0)}}$ that satisfies $0<t^0_{(0)}<t^1_{(0)}<t^2_{(0)}<\ldots<t^{m_{(0)}}_{(0)}<1$.
Next we solve \eqref{prob_semisemi2} using $\{t_{(0)}^i\}_{i=1}^{m_{(0)}}$ instead of 
$\{\bar x_h^i\}_{i=1}^m$ by a semi-smooth Newton method, cf. \cite{Ulbrich}. 
Note that \eqref{prob_semisemi2} is a finite-dimensional problem of dimension $m_{(0)}+1$, independently of $h$. 
This yields $(q_h^{(0)},u_h^{(0)},z_h^{(0)})$. We compute the roots $\{t_{(1)}^i\}_{i=1}^{m_{(1)}}\subset(0,1)$
of $z_h^{(0)}$ and solve \eqref{prob_semisemi2} using $\{t_{(1)}^i\}_{i=1}^{m_{(1)}}$ instead of 
$\{\bar x_h^i\}_{i=1}^m$ to obtain $(q_h^{(1)},u_h^{(1)},z_h^{(1)})$. This process is iterated.
We call the step of the algorithm where the new estimate $\{t_{(k+1)}^i\}_{i=1}^{m_{(k+1)}}$ of $\{\bar x_h^i\}_{i=1}^m$ is obtained, the \emph{outer} iteration. 
The \emph{inner} iteration consists of solving \eqref{prob_semisemi2}.
The outer iteration and thereby the overall algorithm are terminated if an approximation $t_{(k)}:=(t_{(k)}^1,\ldots,t_{(k)}^{m_{(k)}})^T$, $k\geq 1$, is obtained that satisfies
\begin{equation}\label{eq_convreq2}\tag{T}
m_{(k)}=m_{(k-1)}  \qquad\text{and}\qquad
\|t_{(k)}-t_{(k-1)}\|_2\leq \epsilon_{\text{out}},
\end{equation}
where  
$\epsilon_{\text{out}}>0$ is some small tolerance, e.g., 
$\epsilon_{\text{out}}=10^{-10}$.

All in all, these considerations give rise to the following algorithm.

\begin{algorithm2e}[H]
	\DontPrintSemicolon
	\caption{Solving the semi-discrete problem \eqref{prob_BVh}}\label{alg_MPDAPz2}
	\KwIn{ $m_{(0)}\in\mathbb{N}\cup\{0\}$, $t_{(0)}\in\R^{m_{(0)}}$ and $\epsilon_{\text{in}},\epsilon_{\text{out}}>0$}
	\For(\tcp*[f]{outer iteration}){$k=0,1,2,\ldots$}
	{	
		\lIf{\eqref{eq_convreq2} holds}{let $m:=m_{(k)}$, $\bar x_h:=t_{(k)}$,
			 and extract $(\bar a_h,\bar c_h)$ from $q_h^{(k)}$; STOP \tcp*[f]{check termination criterion}}
		Obtain $(q_h^{(k)},u_h^{(k)},z_h^{(k)})$ by solving
		\eqref{prob_semisemi2} to tolerance $\epsilon_{\text{in}}$ 
		\tcp*[f]{inner iter.}\label{line_ssn} \\
		Compute the roots $t_{(k+1)}\in\R^{m_{(k+1)}}$ of $z_h^{(k)}$\tcp*[f]{next approximation}\label{line_rootsadjst2}
	}
	\KwOut{ $\bar x_h\in\R^m$, $(\bar a_h,\bar c_h)\in\R^{m+1}$}
\end{algorithm2e}

While it is theoretically possible that the inner iteration does not converge, we did not observe divergence in the numerical experiments that we carried out.
However, we did sometimes observe cycling of the outer iteration, e.g., $t_{(2k+2)}=t_{(2k)}$ and $t_{(2k+3)}=t_{(2k+1)}$ for all $k$ sufficiently large. 
Since this did only occur for iterates with an equal number of roots of the adjoint state, the following modification of line~\ref{line_rootsadjst2} was possible and turned out to be sufficient: Compute the roots $t_{(k+1)}$ in line~\ref{line_rootsadjst2}, and if
$\|t_{(k+1)}-t_{(k)}\|_2\geq\|t_{(k)}-t_{(k-1)}\|_2$,
then use $0.5 t_{(k)}+0.5 t_{(k+1)}$ instead of $t_{(k+1)}$ as new approximation of $\{\bar x_h^i\}_{i=1}^m$.

For the numerical computations we use 
\begin{equation}\label{eq_paramchoice2}
a_h^{(0)} := 0, \quad
m_{(0)}:= 0, \quad t_{(0)}:=\{ \}, 
\quad \epsilon_{\text{out}}:=10^{-10},
\quad\text{and}\quad
\epsilon_{\text{in}}:=10^{-12}.
\end{equation}
We stress that our intent is to display the order of convergence, hence the parameter choices are made in such a way that the computed solutions are highly accurate.

\subsection{Optimization algorithm for full discretization}

The algorithm that we use to solve \eqref{prob_BVch} is very similar to Algorithm~\ref{alg_MPDAPz2}. In fact, there are only two differences: The approximating points $\{t_{(k)}^i\}_{i=1}^{m_{(k)}}$ have to be gridpoints and, in view of our theoretical findings from \Cref{cor:fullydiscjumps}, we may add two gridpoints for every root of $z_h^{(k)}$. To meet these demands we first compute the roots of $z_h^{(k)}$ in the same way as in Algorithm~\ref{alg_MPDAPz2}. Subsequently, every root is replaced by the two gridpoints adjacent to that root, except if a root happens to be on a gridpoint, in which case only that gridpoint is used. This is in agreement with \Cref{cor:fullydiscjumps}. Indeed, if a gridpoint
is added at which no jump occurs, then the inner iteration accounts for this by yielding zero for the corresponding coefficient (recall the representation \eqref{eq:barqhpdarstellung1}).
Since these are the only changes in Algorithm~\ref{alg_MPDAPz2}, we do not state the resulting algorithm. In the numerical experiments we use the same set of parameters as for Algorithm~\ref{alg_MPDAPz2}, cf.~\eqref{eq_paramchoice2}.


\subsection{Example 1: Known Solution} \label{sec:academicaexample}

We construct an example by defining the following quantities:
\begin{itemize}
	\item $c := 12-4\sqrt 8$, $x_c := \frac{1}{2\pi} \arccos(\frac{c}{4})$;
	\item $\alpha := 10^{-5}$;
	\item $\bar q := 0.5 + 1_{(x_c,1)} - 2 \cdot 1_{(0.5, 1)} + 1.5 \cdot  1_{(1-x_c,1)}$;
	\item $\bar u := S(\bar q)$, $a(x) := 1$, $d_0(x) := 0$;
	\item $\bar \Phi(x) := \frac{\alpha}{2c} \left[ (1-\cos(4\pi x)) - c(1-\cos(2\pi x)) \right]$ (a linear combination of a wave with two positive peaks and one negative peak; the peaks are not equidistant throughout $\Omega$, but symmetrical to $0.5$);
	\item $\bar z := \bar \Phi^\prime$;
	\item $u_d := \bar u + \bar z^{\prime\prime}$.
\end{itemize}
It is straightforward to check that these quantities satisfy the conditions from \Cref{thm_optcond}. In particular, given this $\alpha$ and this $u_d$ the exact solution to \eqref{prob:BV} is $\bar q$. 
The approximated solutions to this problem are depicted in \Cref{fig_ex2_plots}.

\setlength\fheight{4cm}
\setlength\fwidth{5.25cm}
\begin{figure}[t]
	\centering
	\begin{subfigure}[t]{0.495\linewidth}
		\centering
%
\begin{tikzpicture}

\begin{axis}[%
width=0.951\fwidth,
height=\fheight,
at={(0\fwidth,0\fheight)},
scale only axis,
xmin=0,
xmax=1,
xmajorgrids,
ymin=0.0,
ymax=3,
ymajorgrids,
axis background/.style={fill=white},
title style={font=\bfseries},
title={Control},
legend style={legend cell align=left,align=left,draw=white!15!black}
]
\addplot [color=black!35!blue,solid,line width=2.0pt,forget plot]
  table[row sep=crcr]{%
0	0.500000009196998\\
0.222557550058686	0.500000009196998\\
};
\addplot [color=black!35!blue,solid,line width=2.0pt,forget plot]
  table[row sep=crcr]{%
0.222557550058686	1.9999999997785\\
0.499999999677575	1.9999999997785\\
};
\addplot [color=black!35!blue,solid,line width=2.0pt,forget plot]
  table[row sep=crcr]{%
0.499999999677575	1.50000001881792\\
0.77744244959869	1.50000001881792\\
};
\addplot [color=black!35!blue,solid,line width=2.0pt,forget plot]
  table[row sep=crcr]{%
0.77744244959869	2.5000000216491\\
1	2.5000000216491\\
};
\end{axis}
\end{tikzpicture}%
		\caption{$\bar q_h$}\label{fig_ex2_q_varD}
	\end{subfigure}
	\hfill
	\begin{subfigure}[t]{0.495\linewidth}
		\centering
		\input{ex2_z_262144_varD}
		\caption{$\bar z_h$}\label{fig_ex2_z_varD}
	\end{subfigure}
	
	\bigskip
	
	\begin{subfigure}[t]{0.495\linewidth}
		\centering
		\input{ex2_u_262144_varD}
		\caption{$\bar u_h$}\label{fig_ex2_u_varD}
	\end{subfigure}
	\hfill
	\begin{subfigure}[t]{0.495\linewidth}
		\centering
		\input{ex2_Phi_262144_varD}
		\caption{$\bar\Phi_h$}\label{fig_ex2_Phi_varD}
	\end{subfigure}
	\caption{Example 1: The semi-discrete solution to the data from \Cref{sec:academicaexample}. The discretization parameter $h$ is roughly $3.8\cdot 10^{-6}$. The inclusions provided in \Cref{cor:discretesupportcondition} are clearly visible.}
	\label{fig_ex2_plots}
\end{figure}

\setlength\fheight{8.5cm}
\setlength\fwidth{13cm}

\Cref{fig_ex2_errs_varD} displays the errors between solutions to the original problem \eqref{prob:BV} and solutions to the variationally discretized problem \eqref{prob_BVh}. We observe that the error estimates of \Cref{thm_errorestimatestatesorderh2} and \Cref{cor_errorestimatestatesorderh2} are indeed sharp. In addition, the $L^2(\Omega)$-error of the controls is \emph{not} of order $h^2$, showing that the derived error estimates for the control are not satisfied for the $L^2(\Omega)$-norm.  
We remark that an error estimate of order $\mathcal{O}(h)$ for the controls with respect to the $\LL$-norm follows easily from \Cref{cor_errorestimatestatesorderh2}.

In \Cref{fig_ex2_errs} we compare the solutions of the fully discretized problem \eqref{prob_BVch} to the solutions of the original problem. Again we find the error estimates from \Cref{prop:suboptimalstateerrorch}, \Cref{prop:adjointconv1} 
and \Cref{cor_errorestimatestatesorderh_impr} to be sharp and the $L^2(\Omega)$-error of the controls to be of lower order than the $L^1(\Omega)$-error. Correspondingly, it is straightforward to deduce an error estimate in $\LL$ of order ${\mathcal O}(h^{\frac12})$ for the controls.
The slightly erratic behavior of the errors can be explained by the fact that on some grids the locations of the jumps of the continuous optimal control $\bar q$ are better resolved by the gridpoints than on others; we stress that the grids are \emph{not} nested.

\begin{figure}[t]
	\centering
	\scalebox{0.85}{ 
%
\definecolor{mycolor1}{rgb}{0.47000,0.53000,0.60000}%
\definecolor{mycolor2}{rgb}{0.69000,0.77000,0.87000}%
\begin{tikzpicture}

\begin{axis}[%
width=0.951\fwidth,
height=\fheight,
at={(0\fwidth,0\fheight)},
scale only axis,
xmode=log,
xmin=6.1038881767686e-05,
xmax=0.0666666666666667,
xminorticks=true,
xlabel={h},
xmajorgrids,
xminorgrids,
ymode=log,
ymin=4.47089410493946e-12,
ymax=0.439472272371565,
yminorticks=true,
ymajorgrids,
yminorgrids,
axis background/.style={fill=white},
legend style={at={(0.97,0.03)},anchor=south east,legend cell align=left,align=left,draw=white!15!black}
]
\addplot [color=black!35!blue,solid,line width=1.0pt,mark=*,mark options={solid,fill=black!35!blue,draw=black}]
  table[row sep=crcr]{%
0.0666666666666667	0.416577895599417\\
0.032258064516129	0.0872055934275276\\
0.0158730158730159	0.0201032126637091\\
0.0078740157480315	0.00538157586987273\\
0.00392156862745098	0.00132144592229385\\
0.00195694716242661	0.000331226615169428\\
0.000977517106549365	8.3559317588293e-05\\
0.000488519785051295	2.11526565664677e-05\\
0.000244200244200244	5.25729076383021e-06\\
0.000122085215480405	1.17320498710111e-06\\
6.1038881767686e-05	3.11410102584297e-07\\
};
\addlegendentry{$\|\bar q-\bar q_h\|_1$};

\addplot [color=mycolor1,solid,line width=1.0pt,mark=square*,mark options={solid,fill=mycolor1,draw=black}]
  table[row sep=crcr]{%
0.0666666666666667	0.439472272371565\\
0.032258064516129	0.110311012975148\\
0.0158730158730159	0.039782220693191\\
0.0078740157480315	0.0180402310616587\\
0.00392156862745098	0.00865431475044991\\
0.00195694716242661	0.00429111896601839\\
0.000977517106549365	0.00214575278366303\\
0.000488519785051295	0.00107579911669172\\
0.000244200244200244	0.000537039925718403\\
0.000122085215480405	0.000259825749021082\\
6.1038881767686e-05	0.000132776117712211\\
};
\addlegendentry{$\|\bar q-\bar q_h\|_2$};

\addplot [color=black!45!red,solid,line width=1.0pt,mark=diamond*,mark options={solid,fill=black!45!red,draw=black}]
  table[row sep=crcr]{%
0.0666666666666667	0.00232292513209488\\
0.032258064516129	0.000528990207837366\\
0.0158730158730159	0.000123816021484203\\
0.0078740157480315	3.44704734399576e-05\\
0.00392156862745098	8.46805426068389e-06\\
0.00195694716242661	2.12897803532405e-06\\
0.000977517106549365	5.39208529690376e-07\\
0.000488519785051295	1.37124509318936e-07\\
0.000244200244200244	3.40207996927957e-08\\
0.000122085215480405	7.29994667401848e-09\\
6.1038881767686e-05	1.97847115866685e-09\\
};
\addlegendentry{$\|\bar u-\bar u_h\|_2$};

\addplot [color=black!45!green,solid,line width=1.0pt,mark=triangle*,mark options={solid,fill=black!45!green,draw=black}]
  table[row sep=crcr]{%
0.0666666666666667	1.85042464580715e-05\\
0.032258064516129	4.55258994263345e-06\\
0.0158730158730159	1.08713882118806e-06\\
0.0078740157480315	3.01735210649152e-07\\
0.00392156862745098	7.36690479861588e-08\\
0.00195694716242661	1.85397737957003e-08\\
0.000977517106549365	4.69392752866103e-09\\
0.000488519785051295	1.19359167776551e-09\\
0.000244200244200244	2.93272241195142e-10\\
0.000122085215480405	5.31378652695149e-11\\
6.1038881767686e-05	1.71013260383225e-11\\
};
\addlegendentry{$\|\bar z-\bar z_h\|_\infty$};

\addplot [color=mycolor2,dashed,line width=1.0pt]
  table[row sep=crcr]{%
6.1038881767686e-05	5.58861763117433e-08\\
0.0666666666666667	0.0666666666666667\\
};
\addlegendentry{${{\mathcal O}}(h^2)$};

\addplot [color=mycolor2,dashed,line width=1.0pt,forget plot]
  table[row sep=crcr]{%
6.1038881767686e-05	4.47089410493946e-10\\
0.0666666666666667	0.000533333333333333\\
};
\addplot [color=mycolor2,dashed,line width=1.0pt,forget plot]
  table[row sep=crcr]{%
6.1038881767686e-05	4.47089410493946e-12\\
0.0666666666666667	5.33333333333333e-06\\
};
\end{axis}
\end{tikzpicture}%
 }
	\caption{Example 1: Convergence plots of the errors of the solutions to the semi-discrete problem \eqref{prob_BVh} compared to the exact solution. The exact solution is known.}
	\label{fig_ex2_errs_varD}
\end{figure}

\begin{figure}[t]
	\centering
	\scalebox{0.85}{ 
%
\definecolor{mycolor1}{rgb}{0.47000,0.53000,0.60000}%
\definecolor{mycolor2}{rgb}{0.69000,0.77000,0.87000}%
\begin{tikzpicture}

\begin{axis}[%
width=0.951\fwidth,
height=\fheight,
at={(0\fwidth,0\fheight)},
scale only axis,
xmode=log,
xmin=3.81471181759574e-06,
xmax=0.0666666666666667,
xminorticks=true,
xlabel={h},
xmajorgrids,
xminorgrids,
ymode=log,
ymin=3.59642505307626e-11,
ymax=0.506439467901868,
yminorticks=true,
ymajorgrids,
yminorgrids,
axis background/.style={fill=white},
legend style={at={(0.97,0.03)},anchor=south east,legend cell align=left,align=left,draw=white!15!black}
]
\addplot [color=black!35!blue,solid,line width=1.0pt,mark=*,mark options={solid,fill=black!35!blue,draw=black}]
  table[row sep=crcr]{%
0.0666666666666667	0.443670677923999\\
0.032258064516129	0.120401369001373\\
0.0158730158730159	0.035002833112647\\
0.0078740157480315	0.0204660740453858\\
0.00392156862745098	0.0102755841472339\\
0.00195694716242661	0.00541532951504797\\
0.000977517106549365	0.00301445792483316\\
0.000488519785051295	0.00182141816045074\\
0.000244200244200244	0.00122675842162951\\
0.000122085215480405	0.000146602915511212\\
6.1038881767686e-05	0.000389983566323224\\
3.05185094759972e-05	0.000119994755337618\\
1.52590218966964e-05	8.29318443921854e-05\\
7.62945273935501e-06	1.54324099092592e-05\\
3.81471181759574e-06	6.16156679467323e-06\\
};
\addlegendentry{$\|\bar q-\hat q_h\|_1$};

\addplot [color=mycolor1,solid,line width=1.0pt,mark=square*,mark options={solid,fill=mycolor1,draw=black}]
  table[row sep=crcr]{%
0.0666666666666667	0.506439467901868\\
0.032258064516129	0.22862733197568\\
0.0158730158730159	0.132513921189624\\
0.0078740157480315	0.120908354098617\\
0.00392156862745098	0.0842292968130032\\
0.00195694716242661	0.0608310318806978\\
0.000977517106549365	0.0448314723871757\\
0.000488519785051295	0.034151586843198\\
0.000244200244200244	0.027304400332061\\
0.000122085215480405	0.0115945630292224\\
6.1038881767686e-05	0.015144909695937\\
3.05185094759972e-05	0.00872245473403043\\
1.52590218966964e-05	0.0070579549117288\\
7.62945273935501e-06	0.00340386069031463\\
3.81471181759574e-06	0.00223753739089957\\
};
\addlegendentry{$\|\bar q-\hat q_h\|_2$};

\addplot [color=black!45!red,solid,line width=1.0pt,mark=diamond*,mark options={solid,fill=black!45!red,draw=black}]
  table[row sep=crcr]{%
0.0666666666666667	0.00214671533240913\\
0.032258064516129	0.000510195520989621\\
0.0158730158730159	0.000122393442828323\\
0.0078740157480315	4.39399893954513e-05\\
0.00392156862745098	1.51111166771188e-05\\
0.00195694716242661	7.51821716749298e-06\\
0.000977517106549365	4.36353832401294e-06\\
0.000488519785051295	2.85715512270395e-06\\
0.000244200244200244	2.10946271557849e-06\\
0.000122085215480405	5.244188306296e-08\\
6.1038881767686e-05	7.06662467866952e-07\\
3.05185094759972e-05	1.922017865198e-07\\
1.52590218966964e-05	1.45387420116688e-07\\
7.62945273935501e-06	1.67679153177858e-08\\
3.81471181759574e-06	5.11658554736582e-09\\
};
\addlegendentry{$\|\bar u-\hat u_h\|_2$};

\addplot [color=black!45!green,solid,line width=1.0pt,mark=triangle*,mark options={solid,fill=black!45!green,draw=black}]
  table[row sep=crcr]{%
0.0666666666666667	1.738928676756e-05\\
0.032258064516129	4.48651559187464e-06\\
0.0158730158730159	1.05066050205553e-06\\
0.0078740157480315	2.99371774385269e-07\\
0.00392156862745098	9.07654124092071e-08\\
0.00195694716242661	4.79184791549297e-08\\
0.000977517106549365	2.7792154456722e-08\\
0.000488519785051295	1.80574928740688e-08\\
0.000244200244200244	1.32731664838513e-08\\
0.000122085215480405	3.36334646695789e-10\\
6.1038881767686e-05	4.43972483500001e-09\\
3.05185094759972e-05	1.20869424209522e-09\\
1.52590218966964e-05	9.15278312052028e-10\\
7.62945273935501e-06	1.07311440070203e-10\\
3.81471181759574e-06	3.59642505307626e-11\\
};
\addlegendentry{$\|\bar z-\hat z_h\|_\infty$};

\addplot [color=mycolor2,dashed,line width=1.0pt]
  table[row sep=crcr]{%
3.81471181759574e-06	9.53677954398935e-06\\
0.0666666666666667	0.166666666666667\\
};
\addlegendentry{${{\mathcal O}}(h)$};

\addplot [color=mycolor2,dashed,line width=1.0pt,forget plot]
  table[row sep=crcr]{%
3.81471181759574e-06	3.05176945407659e-08\\
0.0666666666666667	0.000533333333333333\\
};
\addplot [color=mycolor2,dashed,line width=1.0pt,forget plot]
  table[row sep=crcr]{%
3.81471181759574e-06	3.05176945407659e-10\\
0.0666666666666667	5.33333333333333e-06\\
};
\end{axis}
\end{tikzpicture}%
 }
	\caption{Example 1: Convergence plots of the errors of the solutions to the fully discrete problem \eqref{prob_BVch} compared to the exact solution. The exact solution is known.}
	\label{fig_ex2_errs}
\end{figure}

\subsection{Example 2: Unknown Solution} \label{sec:naturalexample}

We consider $\alpha := 10^{-5}$ and $u_d(x) := 0.5\pi^{-2} (1-\cos(2\pi x))$. An approximate solution to \eqref{prob:BV} is shown in \Cref{fig_ex3_plots}.

\setlength\fheight{4cm}
\setlength\fwidth{5.25cm}
\begin{figure}[t]
	\centering
	\begin{subfigure}[t]{0.495\linewidth}
		\centering
%
\begin{tikzpicture}

\begin{axis}[%
width=0.951\fwidth,
height=\fheight,
at={(0\fwidth,0\fheight)},
scale only axis,
xmin=0,
xmax=1,
xmajorgrids,
ymin=-1,
ymax=1.5,
ymajorgrids,
axis background/.style={fill=white},
title style={font=\bfseries},
title={Control}
]
\addplot [color=black!35!blue,solid,line width=2.0pt,forget plot]
  table[row sep=crcr]{%
0	-0.58747362502274\\
0.291511545096475	-0.58747362502274\\
};
\addplot [color=black!35!blue,solid,line width=2.0pt,forget plot]
  table[row sep=crcr]{%
0.291511545096475	0.881173376024765\\
0.708488457217153	0.881173376024765\\
};
\addplot [color=black!35!blue,solid,line width=2.0pt,forget plot]
  table[row sep=crcr]{%
0.708488457217153	0.29369976476994\\
1	0.29369976476994\\
};
\end{axis}
\end{tikzpicture}%
		\caption{$\bar q_h$}\label{fig_ex3_q_varD}
	\end{subfigure}
	\hfill
	\begin{subfigure}[t]{0.495\linewidth}
		\centering
		\input{ex3_z_262144_varD}
		\caption{$\bar z_h$}\label{fig_ex3_z_varD}
	\end{subfigure}
	
	\bigskip
	
	\begin{subfigure}[t]{0.495\linewidth}
		\centering
		\input{ex3_u_262144_varD}
		\caption{$\bar u_h$}\label{fig_ex3_u_varD}
	\end{subfigure}
	\hfill
	\begin{subfigure}[t]{0.495\linewidth}
		\centering
		\input{ex3_Phi_262144_varD}
		\caption{$\bar\Phi_h$}\label{fig_ex3_Phi_varD}
	\end{subfigure}
	\caption{Example 2: The variationally discrete solution to the data from \Cref{sec:naturalexample}. The discretization parameter $h$ is roughly $3.8\cdot 10^{-6}$. The inclusions provided in \Cref{cor:discretesupportcondition} are clearly visible.}
	\label{fig_ex3_plots}
\end{figure}

First we turn to the variationally discrete problem. As we do not have a known solution, we compute a reference solution $(\bar u_{h_{\text{ref}}}, \bar q_{h_{\text{ref}}}, \bar z_{h_{\text{ref}}}, \bar \Phi_{h_{\text{ref}}})$ on a fine grid, more specifically $h_{\text{ref}} \approx 9.5\cdot 10^{-7}$, and approximate the errors via $\lVert \bar q_h - \bar q \rVert_{L^1(\Omega)}\approx \lVert \bar q_h - \bar q_{h_{\text{ref}}} \rVert_{L^1(\Omega)}$. The same is done for the states and the adjoint states. \Cref{fig_ex3_errs_varD} displays the approximated errors. As in Example~1 we observe that the rates from \Cref{thm_errorestimatestatesorderh2} and \Cref{cor_errorestimatestatesorderh2} are sharp and that the $L^2(\Omega)$-error of the control is of lower order than the $L^1(\Omega)$-error.

\setlength\fheight{8.5cm}
\setlength\fwidth{13cm}

\begin{figure}[t]
	\centering
	\scalebox{0.85}{ 
%
\definecolor{mycolor1}{rgb}{0.47000,0.53000,0.60000}%
\definecolor{mycolor2}{rgb}{0.69000,0.77000,0.87000}%
\begin{tikzpicture}

\begin{axis}[%
width=0.951\fwidth,
height=\fheight,
at={(0\fwidth,0\fheight)},
scale only axis,
xmode=log,
xmin=6.1038881767686e-05,
xmax=0.0666666666666667,
xminorticks=true,
xlabel={h},
xmajorgrids,
xminorgrids,
ymode=log,
ymin=7.82406468364406e-12,
ymax=0.186357914226829,
yminorticks=true,
ymajorgrids,
yminorgrids,
axis background/.style={fill=white},
legend style={at={(0.97,0.03)},anchor=south east,legend cell align=left,align=left,draw=white!15!black}
]
\addplot [color=black!35!blue,solid,line width=1.0pt,mark=*,mark options={solid,fill=black!35!blue,draw=black}]
  table[row sep=crcr]{%
0.0666666666666667	0.0425429873415083\\
0.032258064516129	0.00972667090893375\\
0.0158730158730159	0.0021144830928374\\
0.0078740157480315	0.000597759718612861\\
0.00392156862745098	0.000127223261615493\\
0.00195694716242661	3.62497518635097e-05\\
0.000977517106549365	8.05728638321459e-06\\
0.000488519785051295	1.98049522529154e-06\\
0.000244200244200244	4.95061723783385e-07\\
0.000122085215480405	1.22806926753047e-07\\
6.1038881767686e-05	2.95117391286834e-08\\
};
\addlegendentry{$\|\bar q-\bar q_h\|_1$};

\addplot [color=mycolor1,solid,line width=1.0pt,mark=square*,mark options={solid,fill=mycolor1,draw=black}]
  table[row sep=crcr]{%
0.0666666666666667	0.186357914226829\\
0.032258064516129	0.0898531452818644\\
0.0158730158730159	0.0428943888521211\\
0.0078740157480315	0.0215866134970462\\
0.00392156862745098	0.010523842582048\\
0.00195694716242661	0.00534593222285767\\
0.000977517106549365	0.00263513675067126\\
0.000488519785051295	0.00131239271761214\\
0.000244200244200244	0.000660251365039529\\
0.000122085215480405	0.00033733022341685\\
6.1038881767686e-05	0.000182216782052868\\
};
\addlegendentry{$\|\bar q-\bar q_h\|_2$};

\addplot [color=black!45!red,solid,line width=1.0pt,mark=diamond*,mark options={solid,fill=black!45!red,draw=black}]
  table[row sep=crcr]{%
0.0666666666666667	0.000964842767655226\\
0.032258064516129	0.000208196500544665\\
0.0158730158730159	5.60181290019243e-05\\
0.0078740157480315	1.24307702896955e-05\\
0.00392156862745098	3.44546319594136e-06\\
0.00195694716242661	7.74342875151469e-07\\
0.000977517106549365	2.08254258624046e-07\\
0.000488519785051295	5.29454222275903e-08\\
0.000244200244200244	1.31762638809924e-08\\
0.000122085215480405	3.26170597498923e-09\\
6.1038881767686e-05	7.98985381704232e-10\\
};
\addlegendentry{$\|\bar u-\bar u_h\|_2$};

\addplot [color=black!45!green,solid,line width=1.0pt,mark=triangle*,mark options={solid,fill=black!45!green,draw=black}]
  table[row sep=crcr]{%
0.0666666666666667	1.11424013529959e-05\\
0.032258064516129	2.30268300542183e-06\\
0.0158730158730159	6.0424163240567e-07\\
0.0078740157480315	1.31937470343381e-07\\
0.00392156862745098	3.64097052690644e-08\\
0.00195694716242661	8.15309945787255e-09\\
0.000977517106549365	2.19577846955133e-09\\
0.000488519785051295	5.57486132479962e-10\\
0.000244200244200244	1.3878665447789e-10\\
0.000122085215480405	3.43669113281905e-11\\
6.1038881767686e-05	8.29703521192842e-12\\
};
\addlegendentry{$\|\bar z-\bar z_h\|_\infty$};

\addplot [color=mycolor2,dashed,line width=1.0pt]
  table[row sep=crcr]{%
6.1038881767686e-05	5.58861763117433e-08\\
0.0666666666666667	0.0666666666666667\\
};
\addlegendentry{${{\mathcal O}}(h^2)$};

\addplot [color=mycolor2,dashed,line width=1.0pt,forget plot]
  table[row sep=crcr]{%
6.1038881767686e-05	5.58861763117433e-10\\
0.0666666666666667	0.000666666666666667\\
};
\addplot [color=mycolor2,dashed,line width=1.0pt,forget plot]
  table[row sep=crcr]{%
6.1038881767686e-05	7.82406468364406e-12\\
0.0666666666666667	9.33333333333333e-06\\
};
\end{axis}
\end{tikzpicture}%
 }
	\caption{Example 2: Convergence plots of the errors of the solutions to the semi-discrete problem \eqref{prob_BVh} compared to an approximation of the exact solution. The reference solution is computed as solution to \eqref{prob_BVh} with $h_{\text{ref}}\approx 3.8\cdot 10^{-6}$.}
	\label{fig_ex3_errs_varD}
\end{figure}

The same procedure is applied to the fully discrete problem, and the results are depicted in \Cref{fig_ex3_errs}. Once again the proven rates turn out to be sharp and the $L^2(\Omega)$-rate is of lower order than the $L^1(\Omega)$-rate. 

\begin{figure}[t]
	\centering
	\scalebox{0.85}{ 
%
\definecolor{mycolor1}{rgb}{0.47000,0.53000,0.60000}%
\definecolor{mycolor2}{rgb}{0.69000,0.77000,0.87000}%
\begin{tikzpicture}

\begin{axis}[%
width=0.951\fwidth,
height=\fheight,
at={(0\fwidth,0\fheight)},
scale only axis,
xmode=log,
xmin=3.81471181759574e-06,
xmax=0.0666666666666667,
xminorticks=true,
xlabel={h},
xmajorgrids,
xminorgrids,
ymode=log,
ymin=2.27876094560555e-10,
ymax=0.395241867836608,
yminorticks=true,
ymajorgrids,
yminorgrids,
axis background/.style={fill=white},
legend style={at={(0.97,0.03)},anchor=south east,legend cell align=left,align=left,draw=white!15!black}
]
\addplot [color=black!35!blue,solid,line width=1.0pt,mark=*,mark options={solid,fill=black!35!blue,draw=black}]
  table[row sep=crcr]{%
0.0666666666666667	0.176403605542882\\
0.032258064516129	0.0117074784693991\\
0.0158730158730159	0.0497009327434055\\
0.0078740157480315	0.00158925111517806\\
0.00392156862745098	0.0123494307519405\\
0.00195694716242661	0.000694827467855047\\
0.000977517106549365	0.00199114109659443\\
0.000488519785051295	0.0012686150928192\\
0.000244200244200244	0.000597888422024386\\
0.000122085215480405	0.000262591908036826\\
6.1038881767686e-05	9.49603409074374e-05\\
3.05185094759972e-05	1.11488361240222e-05\\
1.52590218966964e-05	3.07553462206679e-05\\
7.62945273935501e-06	2.01665038695403e-05\\
3.81471181759574e-06	9.69847719249151e-06\\
};
\addlegendentry{$\|\bar q-\hat q_h\|_1$};

\addplot [color=mycolor1,solid,line width=1.0pt,mark=square*,mark options={solid,fill=mycolor1,draw=black}]
  table[row sep=crcr]{%
0.0666666666666667	0.395241867836608\\
0.032258064516129	0.100022302812705\\
0.0158730158730159	0.207532293743891\\
0.0078740157480315	0.0382380168744604\\
0.00392156862745098	0.105375159121725\\
0.00195694716242661	0.024930724615539\\
0.000977517106549365	0.0422831901483732\\
0.000488519785051295	0.0337437878152851\\
0.000244200244200244	0.02316722718401\\
0.000122085215480405	0.015355726709019\\
6.1038881767686e-05	0.00923835334556517\\
3.05185094759972e-05	0.00318183252600583\\
1.52590218966964e-05	0.00524276243116127\\
7.62945273935501e-06	0.00426818678899635\\
3.81471181759574e-06	0.00296898449806851\\
};
\addlegendentry{$\|\bar q-\hat q_h\|_2$};

\addplot [color=black!45!red,solid,line width=1.0pt,mark=diamond*,mark options={solid,fill=black!45!red,draw=black}]
  table[row sep=crcr]{%
0.0666666666666667	0.00100613458696653\\
0.032258064516129	0.000191856475296835\\
0.0158730158730159	0.000142282240031746\\
0.0078740157480315	1.21123038869394e-05\\
0.00392156862745098	3.77668187754961e-05\\
0.00195694716242661	2.3356022938396e-06\\
0.000977517106549365	6.08831267517538e-06\\
0.000488519785051295	3.9005672057342e-06\\
0.000244200244200244	1.83838238976472e-06\\
0.000122085215480405	8.08793211686803e-07\\
6.1038881767686e-05	2.94373272461274e-07\\
3.05185094759972e-05	3.73713433464521e-08\\
1.52590218966964e-05	9.13692866481473e-08\\
7.62945273935501e-06	6.4945629418375e-08\\
3.81471181759574e-06	3.29652881373877e-08\\
};
\addlegendentry{$\|\bar u-\hat u_h\|_2$};

\addplot [color=black!45!green,solid,line width=1.0pt,mark=triangle*,mark options={solid,fill=black!45!green,draw=black}]
  table[row sep=crcr]{%
0.0666666666666667	1.1227754603964e-05\\
0.032258064516129	2.10835066465885e-06\\
0.0158730158730159	1.1931718141091e-06\\
0.0078740157480315	1.31058144993675e-07\\
0.00392156862745098	2.85508900497658e-07\\
0.00195694716242661	1.47735656710123e-08\\
0.000977517106549365	4.4051161427987e-08\\
0.000488519785051295	2.73389765825741e-08\\
0.000244200244200244	1.28980886484064e-08\\
0.000122085215480405	5.6754197049556e-09\\
6.1038881767686e-05	2.06353416784886e-09\\
3.05185094759972e-05	2.57523547209322e-10\\
1.52590218966964e-05	6.54716969139315e-10\\
7.62945273935501e-06	4.51961421531942e-10\\
3.81471181759574e-06	2.27876094560555e-10\\
};
\addlegendentry{$\|\bar z-\hat z_h\|_\infty$};

\addplot [color=mycolor2,dashed,line width=1.0pt]
  table[row sep=crcr]{%
3.81471181759574e-06	7.62942363519148e-06\\
0.0666666666666667	0.133333333333333\\
};
\addlegendentry{${{\mathcal O}}(h)$};

\addplot [color=mycolor2,dashed,line width=1.0pt,forget plot]
  table[row sep=crcr]{%
3.81471181759574e-06	3.05176945407659e-08\\
0.0666666666666667	0.000533333333333333\\
};
\addplot [color=mycolor2,dashed,line width=1.0pt,forget plot]
  table[row sep=crcr]{%
3.81471181759574e-06	4.57765418111489e-10\\
0.0666666666666667	8e-06\\
};
\end{axis}
\end{tikzpicture}%
 }
	\caption{Example 2: Convergence plots of the errors of the solutions to the fully discrete problem \eqref{prob_BVch} compared to an approximation of the exact solution. The reference solution is computed as solution to \eqref{prob_BVch} with $h_{\text{ref}}\approx 2.4\cdot 10^{-7}$.}
	\label{fig_ex3_errs}	
\end{figure}


\subsection{Acknowledgments}

Dominik Hafemeyer acknowledges support from the graduate program TopMath of the Elite Network of Bavaria and the TopMath Graduate Center of TUM Graduate School at Technische Universität München. He is a scholar of the Studienstiftung des deutschen Volkes. Dominik Hafemeyer and Florian Mannel receive support from the IGDK Munich-Graz.


\FloatBarrier

\bibliographystyle{AIMS}

\providecommand{\href}[2]{#2}
\providecommand{\arxiv}[1]{\href{http://arxiv.org/abs/#1}{arXiv:#1}}
\providecommand{\url}[1]{\texttt{#1}}
\providecommand{\urlprefix}{URL }

\medskip
Received xxxx 20xx; revised xxxx 20xx.
\medskip

\end{document}